\documentclass[reqno,11pt]{amsart}
\usepackage{amssymb,amsthm,amsmath,amsfonts,mathrsfs}
\usepackage{graphicx}
\usepackage[top=1in,bottom=1in,left=1in,right=1in]{geometry}
\usepackage[latin1]{inputenc}
\usepackage[T1]{fontenc}
\usepackage{hyperref}
\usepackage[english]{babel}
\usepackage{color}
\usepackage{esint}

\newtheorem{prop}{Proposition}[section]
\newtheorem{lem}{Lemma}[section]

\newtheorem{conj}{Conjecture}[section]
\newtheorem{thm}{Theorem} 
\newtheorem{remark}{Remark}[section]

\newtheorem{defi}{Definition}[section]
\numberwithin{equation}{section}

\newcommand{\be}{\begin{equation} \label}
\newcommand{\ee}{\end{equation}}
\newcommand{\R}{\mathbb{R}}

\newcommand{\eps}{\varepsilon}
\newcommand{\vap}{\varphi}

\newcommand{\ts}{\textstyle}

\begin{document}

\title[Entire and ancient solutions of diffusive Hamilton-Jacobi equations]
{CLASSIFICATION OF ENTIRE AND ANCIENT SOLUTIONS \\
				OF THE DIFFUSIVE HAMILTON-JACOBI EQUATION}
				
\author[CHABI and SOUPLET]{Loth Damagui CHABI and Philippe SOUPLET}

		\address{Universit\'e Sorbonne Paris Nord \& CNRS UMR 7539, 
		Laboratoire Analyse G\'eom\'etrie et Applications, 93430 Villetaneuse,  France.}
	\email{chabi@math.univ-paris13.fr; souplet@math.univ-paris13.fr}

\begin{abstract}
We  study the Liouville type classification and symmetry properties,
in $\R^n$ and in a half-space with Dirichlet boundary conditions,
for entire and ancient solutions of the diffusive Hamilton-Jacobi equation
$$u_t-\Delta u=|\nabla u|^p \eqno(E)$$
 with $p>1$, which arises in optimal stochastic control, 
 in KPZ type models of surface growth and in studies of gradient blow-up.

\vskip 1pt

$\bullet${\hskip 1.5mm}First, we show the optimal Liouville type theorem in $\R^n$:
any ancient solution  in $\R^n$ 
 with sublinear upper growth  at infinity
is necessarily constant. 
This result solves a long standing open problem.

\vskip 1pt

$\bullet${\hskip 1.5mm}Next we completely classify entire solutions in a half-space for $p>2$:
any entire solution is stationary and one-dimensional. 
The assumption $p>2$ is sharp, in view of explicit examples for $p=2$.

\vskip 1pt

$\bullet${\hskip 1.5mm}Then we show that the situation is also completely different for ancient solutions in a half-space:
there 
 exist nonstationary ancient solutions for all $p>1$.
 Nevertheless, we 
 show that any ancient solution is necessarily positive, and that
stationarity and one-dimensionality are recovered provided a -- close to optimal --
polynomial growth restriction 
is imposed on the solution.

\vskip 1pt

$\bullet${\hskip 1.5mm}In addition we establish new and optimal, local estimates of Bernstein and Li-Yau type.
 Remarkably, whereas the Bernstein estimate is valid for all $p>1$, the~Li-Yau estimate is true if and only if $p\ge 2$.

\vskip 1pt

The 
constancy, stationarity and one-dimensionality
are proved using Bernstein and Li-Yau type estimates,
integral estimates, a translation-compactness procedure and comparison arguments.


\end{abstract}

 \maketitle

 \vskip -5mm
 
\section{Introduction}
	 
	 \subsection{Background on the diffusive Hamilton-Jacobi equation} \label{intro1}
	
The diffusive Hamilton-Jacobi equation
\be{eqE0}
u_t-\Delta u=|\nabla u|^p,
\ee
where $p>1$, has a rich background. 
First of all, let us recall that the corresponding Cauchy-Dirichlet problem
\begin{equation}\label{eqE1b}
	\begin{cases}
	u_t-\Delta u=|\nabla u|^p,&x\in\Omega,\ t>0,\\
	u=0,&x\in \partial\Omega,\ t>0,\\
	u(0,x)=u_0(x),&x\in \Omega,
	\end{cases}
	\end{equation}
where $\Omega\subset\R^n$ is a smooth bounded domain,
arises in stochastic control. Namely, for $u_0\in C_0(\overline\Omega)$,
problem~\eqref{eqE1b} has a unique global viscosity solution, which gives the value function of the optimal control problem associated with the stochastic differential system $dX_s =\alpha_s ds+dW_s$, with cost function $|\alpha_s|^{p/(p-1)}$
and zero pay-off at the exit time.
Here $(W_s)_{s>0}$ is a standard Brownian motion, 
 $\alpha_s$ is the control, and $u_0$ represents the distribution of rewards
(see \cite{BaLio04} and cf.~\cite{LL, BaLio04, BarCD, FlMS, AttSou} for more details).
As another motivation,  \eqref{eqE0} corresponds to the so-called deterministic KPZ equation
which, along with its stochastic version,
arises in a well-known model of surface growth by ballistic deposition (cf.~\cite{KPZ86, KS88}
 and see \cite{Ha1, Ha2} for far-reaching developments in the stochastic case).
Also, it can be seen as one the simplest model parabolic problems with first order nonlinearity and, from the point of view of nonlinear parabolic theory, it is thus important to understand its properties 
(cp.~for instance with the extensively studied equation with zero order nonlinearity $u_t-\Delta u= u^p$; see \cite{QSb19}).

\vskip 2pt

Although \eqref{eqE1b}  admits a global viscosity solution, the solution need not be classical in general.
In fact, 
for sufficiently smooth initial data, say 
$u_0\in X:=\bigl\{u_0\in C^1(\bar\Omega);\ u_0=0\hbox{ on $\partial\Omega$}\bigr\}$,
problem \eqref{eqE1b} has a unique maximal, classical solution $u$.
We denote by $T=T(u_0)\in (0,\infty]$ its existence time 
(and $u$ coincides with the unique global viscosity solution on $(0,T)$).
When $p>2$, it is known that, if $u_0\ge 0$ is suitably large, then the solution blows up
in finite time, i.e., $T(u_0)<\infty$, in which case
\be{defgbu}
\lim_{t\to T_-}\|\nabla u(t)\|_\infty=\infty,
\ee
whereas all solutions are global for $p\in (1,2]$. It is also known that the function
itself remains bounded, while its spatial gradient is the quantity to become unbounded.
Moreover, as a consequence of interior gradient estimates \cite{SZ}, 
it is known that singularities can occur only on the boundary $\partial\Omega$.
The blowup phenomenon that occurs for solutions of \eqref{eqE1b}  is usually referred to as 
(boundary) gradient blowup (GBU). 
The study of GBU phenomena for \eqref{eqE1b}  has attracted a lot of attention in the past decades.
 Results include blowup criteria
 \cite{Alaa, ABG, Sou, HM}, blowup locations \cite{Esteve19, LS10, SZ}, blowup rates
 \cite{CG96,GuoHu,ZL13,QSb19,PS_jmpa18,AttSou,MS1},
 blowup profiles \cite{ARS04, CG96, SZ, PS_imrn16, PS_jmpa18, PS_aihp17, MS1},
 continuation after GBU \cite{BaLio04, PZ_aihp12, PS_aihp17, QR18, PS_jmpa18, FPS2020, MS2}, 
 infinite time GBU \cite{SZ, SV}. 
 For other issues concerning equation \eqref{eqE0}, see for instance
  \cite{BSW, Gil, BPT, BQR, GMT} (large time behavior of global solutions and ergodic problem),  
   \cite{And, CC, BSW, GilGK, CS, BVD13, BVD, CiG, Cir} (regularity, rough initial data and initial traces),
 \cite{PZ_aihp12} (controllability), 
 \cite{BLSS, ILS} (extinction), 
 \cite{Att, QR18, AttB, AttSou1} (more general diffusions),
 as well as the monographs  \cite{BarCD, FlMS} for general references.
  We stress that there is a huge literature on \eqref{eqE0}  and related equations and that our list of references is by no means complete.
 
\vskip 2pt

	 \subsection{Background on Liouville type theorems} \label{intro2}

 Liouville type nonexistence, classification and symmetry (or rigidity) theorems are a central issue 
in the study of nonlinear elliptic and parabolic problems,
as well as in other nonlinear PDE's 
 (see,~e.g.,~\cite{Per, DS}, \cite{MM1, MM2, MZ09}, \cite{KNSS, SeSh} and the references therein,
 for geometric  flows, dispersive equations or the Navier-Stokes equations, respectively).
Beside their intrinsic interest, they have, in conjunction with rescaling methods, many applications for the description of the qualitative behavior of solutions (a priori estimates, space and/or time decay, blow-up 
asymptotics, etc.).
The most relevant cases involve the whole space $D=\R^n$ and the half-space 
$D=\R^n_+:=\{x\in\R^n;\, x_n>0\}$ as spatial domains and, 
in the case of  evolution problems, the so-called {\it entire} (or eternal) and {\it ancient} solutions,
 i.e.~solutions defined on $D\times\R$ or $D\times(-\infty,0)$, respectively.
 In order to place our results in perspective,
let us review a number of related, known Liouville-type non-existence, classification and symmetry results
for elliptic and parabolic problems.

	 \subsubsection{Elliptic equations with zero order nonlinearities}
	 First recall the following classical results
 for the model problem with zero order nonlinearity, where $p>1$:
\be{ellclass}
-\Delta u= u^p.
\ee
	
	\begin{itemize}
\item[$\bullet$] Equation \eqref{ellclass} has no positive classical solution in $\R^n$ if (and only if) 
$p<p_S:=(n+2)/(n-2)_+$ \cite{GS1}.

\item[$\bullet$] If $p=p_S$ then, up to dilations and translations, the only positive classical solution of \eqref{ellclass} in $\R^n$ 
is the Aubin-Talenti bubble $u(x)=c(n)(1+|x|^2)^{2-n\over 2}$ \cite{CGS}.

\item[$\bullet$] If $p\le p_S$, then \eqref{ellclass} has no positive classical solution in $\R^n_+$ with zero boundary conditions \cite{GS2}.

	\end{itemize}
For the half-space problem with a more general nonlinearity $f\in C^1([0,\infty))$:
\be{ellclass2}
-\Delta u= f(u)\ \hbox{ in $\R^n_+$, \ with $u=0$ on $\partial\R^n_+$,}
\ee
the following famous BCN conjecture about symmetry, or one-dimensionality was made in \cite{BCN}:

\goodbreak

 \begin{conj} \label{conjBCN}
(i) If there is a positive bounded solution $u$ of \eqref{ellclass2}, then $u$ depends only on $x_n$
(and it follows that $u_{x_n}>0$ and $f(\sup u)= 0$).
\vskip 2pt

(ii) In particular if $f>0$ on $(0,\infty)$, then \eqref{ellclass2} 
 has no positive bounded solution.
 \end{conj}

	\begin{itemize}
\item[$\bullet$] Conjecture~\ref{conjBCN}(i) was proved in \cite{BCN} for $n=2$ and for $n=3$ under the additional hypothesis $f(0)\ge 0$, and then in \cite{DF} (see also \cite{FV}) for $n\le 11$ provided $f\ge 0$ on $[0,s_0)$ and $f\le 0$ on $(s_0,\infty)$ for some $s_0\in(0,\infty]$. The case $n\ge 12$ is essentially open.

\item[$\bullet$] Note that the boundedness assumption cannot be completely removed
in view of counter-examples such as $u=x_n e^{x_1}$, which solves \eqref{ellclass2} for $n\ge 2$ and $f(u)=-u$.
	
\item[$\bullet$] As for Conjecture~\ref{conjBCN}(ii), it was proved in \cite{CLZ} in all dimensions under the additional assumption
that $f\in C^2((0,\infty))$ is convex (which covers in particular the case $f(s)=s^p$ with any $p>1$),
and the boundedness assumption was then weakened to boundedness on finite strips in \cite{DSS}.
	\end{itemize}

	 \subsubsection{Nonlinear heat equation}
\noindent For the parabolic counterpart of \eqref{ellclass}, namely the nonlinear heat equation
\be{parabclass}
u_t-\Delta u= u^p,
\ee
with $p>1$, we have the following results for entire and ancient solutions.

	\begin{itemize}
\item[$\bullet$] If $p<p_S$, then equation \eqref{parabclass} has no nontrivial classical solution $u\ge 0$ in $\R^n\times\R$ \cite{Qu}.
Applications of this (and previous) Liouville type results to universal a priori estimates and initial and final blow-up rates
can be found in \cite{BV,PQS,Qu}.

	\item[$\bullet$] The assumption $p<p_S$ in the preceding item is optimal in view of the existence of positive steady states for $p\ge p_S$.
 A natural question is then:
\be{questionPQ}
\hbox{Are all entire solutions stationary ?}
\ee
This has been investigated in \cite{FY,PQ} and the situation crucially depends on $p$. Namely, letting $p_L=1+6/(n-10)_+$,
the answer is negative for $p\in(p_S,p_L)$,
as there exist bounded positive (radial) entire solutions such that $\|u(t)\|_\infty\to 0$ as $t\to\pm\infty$ \cite{FY}.
On the contrary, for $p>p_L$, any radially symmetric, bounded positive entire solution is stationary
(see \cite{PQ}; but this is not true in general without the radial assumption).
We refer to  \cite{Ta} for related results in the critical case $p=p_S$ and to
\cite{BPQ, Po} for results on sign-changing entire solutions.

\item[$\bullet$] Concerning the half-space case, if $p<p_S$, then \eqref{parabclass} has no nontrivial classical solution $u\ge 0$ in $\R^n_+\times\R$ with zero boundary conditions, and this remains true for bounded solutions if $p<(n+1)/(n-3)_+$ \cite{PQS, Qu}.

	\item[$\bullet$] As for ancient solutions of  \eqref{parabclass}, it was shown in \cite{MZ98} (see also \cite{MZ00,Qu}) that
	if $p<p_S$ and $u\ge 0$ is a classical solution in $\R^n\times(-\infty,0)$,
then $u$ depends only on $t$.
Applications of this classification result to the description of GBU asymptotics were developed in \cite{MZ98, GuoSou}.
On the other hand, for suitable ranges of $p>p_S$, a classification of radial ancient solution was obtained in \cite{PQ},
which shows that they are either stationary or connecting orbits (after passing to similarity variables);
 see also \cite{PY, SWZ} for related results.
Recall that an important class of ancient solutions are the backward self-similar solutions, 
of the form $|t|^{-1/(p-1)}\phi(x/\sqrt{|t|})$, which where studied in \cite{BQ, Le1, Le2} (see also 
\cite[Section~53a]{QSb19} and the references therein).
		\end{itemize}

	 \subsubsection{Diffusive Hamilton-Jacobi equation}
\noindent  Now passing to our main topic, namely equation \eqref{eqE0}, 
we first consider its stationary version, i.e.~the elliptic Hamilton-Jacobi equation,  with $p>1$:
\be{ellHJ}
-\Delta u= |\nabla u|^p.
\ee

	\begin{itemize}
\item[$\bullet$] If $u$ is a classical solution of \eqref{ellHJ} in $\R^n$, 
then $u$ is constant \cite{Lions85}.
	
\item[$\bullet$] If $u$ is a classical solution of \eqref{ellHJ} in $\R^n_+$ with zero boundary conditions,
then $u$ depends only on the variable $x_n$ (see \cite{FPS2020} for $p>2$ and, for $p\in(1,2]$,
see \cite{PV} for bounded solutions and \cite{PS25} for the general case).
Applications of this classification result to the description of GBU asymptotics for  the Dirichlet problem \eqref{eqE1b} 
with $p>2$ in bounded domains were also developed in \cite{FPS2020}.

	\end{itemize}
	
\noindent Concerning whole space ancient solutions of \eqref{eqE0}, the following results are known:
	
		\begin{itemize}
		
\item[$\bullet$]If $p>1$ and  $u$ is a classical solution of  \eqref{eqE0}
in $Q=\R^n\times (-\infty,0)$
such that\footnote{Throughout this article, for a given function $\psi=\psi(x,t)\ge 0$ on $Q$, by the notation ``$\,\cdots\le o(\psi)$ as $\psi\to \infty$'',
we mean that $\dots\le h\circ\psi$ for some locally bounded function 
$h: [0,\infty) \to [0,\infty)$ such that $\lim_{r\to \infty} r^{-1}h(r)=0$.}
\be{hypSZ}
|u(x, t)| = o(|x|^{m} + |t|^{\frac{1}{p}}),\ \hbox{as $|x|+|t|\to\infty$, \ with $m = 1\wedge\frac{1}{p-1}$,}
\ee
then $u$ is constant \cite{SZ}. 
Partial improvements were recently obtained in the case $p>2$ in \cite{Cir},
replacing assumption \eqref{hypSZ} by 
\be{hypCir}
\sup_{(x,t),\, (y,s) \in Q } \frac{|u(x,t) - u(y,s)|}{|x-y|^m + |t-s|^q + 1} < \infty,
\ \hbox{for some $m \in (0,1)$, $q>0$}
\ee
(see~\cite[Theorem~4.1 and Remark~4.2]{Cir}),
and in the case $1<p\le 2$ in \cite{CGM}, replacing assumption \eqref{hypSZ} by 
\be{hypCGM}
u(x, t)\le o(|x|),\ \hbox{as $|x|\to\infty$ \ \ (uniformly in $t<0$);}
\ee
see \cite[Theorem 5.5 and Remark 5.7]{CGM}, and see also \cite{Kruz67,CGM} for related results in the case of subquadratic first-order equations.

 \vskip 2pt
 
\item[$\bullet$]The growth assumptions cannot be completely removed
in view of the existence of the special (entire) solution $u=t+x_1$.
However we observe that there is a gap between this special solution and the assumptions in \eqref{hypSZ}, \eqref{hypCir} and \eqref{hypCGM}. 
In particular, for $p>2$ we have $m=1/(p-1)<1$ in \eqref{hypSZ} whereas,
 although \eqref{hypCir} is close to optimal in the spatial direction and only mildly restrictive in the time direction,
it has the drawback of being a H\"older type condition
and not a pure growth condition like \eqref{hypSZ}. 
As for the range $1<p\le 2$, \eqref{hypCGM} has the advantage as compared with \eqref{hypSZ} of being a one-sided condition.
However, while both reach the sharp value $m=1$ in the space direction, none are sharp in the time direction 
(and \eqref{hypCGM} is more restrictive, by not allowing any growth in time).
\smallskip

The following has thus remained a long standing open problem:
\be{OP}
\begin{aligned}
&\hbox{What is the optimal growth condition to ensure the}\\
&\hbox{constancy of ancient solutions of \eqref{eqE0} in $\R^n\times (-\infty,0)$~?}
\end{aligned}
\ee
	\end{itemize}

	 \subsection{Aims of the paper}
\noindent The goal of this paper is twofold:

\vskip 2pt

		\begin{itemize}
\item[$\bullet$]
To give a complete answer to the open question in \eqref{OP} concerning ancient solutions of \eqref{eqE0} in the whole space.
 
 \vskip 2pt
 
\item[$\bullet$] To  classify entire and ancient solutions of \eqref{eqE0} in the half-space.
Indeed this does not
seem to have been  addressed so far. 
In connection with the BCN conjecture,  with the stationarity question
in \eqref{questionPQ} for the nonlinear heat equation, and with the known classification results for the stationary 
Hamilton-Jacobi equation, we will be especially interested in the questions of unidimensionality and stationarity. 

	\end{itemize}
	
	  \section{Main results}
	
  \subsection{Optimal Liouville type theorem in the whole space}

Our first main result is the following {\it optimal} Liouville type theorem for ancient solutions of \eqref{eqE0}
 in the whole space.

\begin{thm}\label{thmLiouvRn0}
Let $p\ge 2$ and let $u\in C^{2,1}(Q)$ be a solution of \eqref{eqE0} in $Q:={\R^n}\times (-\infty,0)$.
Assume that there exists $t_0\in(-\infty,0)$
such that
\begin{equation} \label{growthassumpt0}
u(x,t_0)\le o(|x|), \quad\hbox{ as $|x|\to\infty$}.
\end{equation}
Then $u$ is constant.
\end{thm}

The optimality is clear from the special traveling wave solution $u(x,t)=t+x_1$
or, more generally, 
\be{TW}
u(x,t)=|a|^pt+a\cdot x,\quad\hbox{with $a\in\R^n\setminus\{0\}$.}
\ee
We also note that \eqref{growthassumpt0} is a  single-time, space only, one-sided condition 
(in particular no lower growth assumption is made).
 For $p\in(1,2)$, we have a similar result, under a space-time sublinearity condition,
which is a bit stronger than \eqref{growthassumpt0} but still essentially optimal in view of the example in~\eqref{TW}.

\begin{thm}\label{thmLiouvRn}
Let $p\in(1,2)$ and let $u\in C^{2,1}(Q)$ be a solution of \eqref{eqE0} in $Q:={\R^n}\times (-\infty,0)$
such that
\begin{equation} \label{growthassumpt}
u(x,t)\le o(|x|+|t|), \quad\hbox{ as $|x|+|t|\to\infty$}.
\end{equation}
Then $u$ is constant.
\end{thm}

Theorems~\ref{thmLiouvRn0}-\ref{thmLiouvRn} solve the open question \eqref{OP}
(cf.~the above-mentioned results in \cite{SZ,Cir, CGM}).
Especially in the case $p\ge 2$, we see that \eqref{growthassumpt0} is the best possible sufficient condition.
 The proofs rely on the improved Bernstein and Li-Yau type estimates in Theorems~\ref{thmLY} and \ref{propBern} below,
  combined with suitable integral estimates and comparison arguments.

\begin{remark} \rm
Theorem~\ref{thmLiouvRn0} was known before only in the special case $p=2$  (see \cite{Hir,Mo}), which reduces to 
(positive solutions of) the heat equation through the transformation $v=e^u$.
 The proof in that case crucially uses the linearity of the (transformed) problem, a feature no longer available for~$p\ne 2$.
\end{remark}

 \subsection{Liouville type classification results in the half-space}
 
 We next turn to the diffusive Hamilton-Jacobi in the half-space:
	\begin{equation}\label{eqE1}
	\begin{cases}
	u_t-\Delta u=|\nabla u|^p,&(x,t)\in Q:=\R^n_+\times I,\\
	\noalign{\vskip 1mm}
	u=0,&(x,t)\in \partial\R^n_+\times I,\\
	\end{cases}
	\end{equation}
	where $\R^n_+:=\{x=(x_1,\cdots, x_n)\in \R^n: x_n>0\}$  and $p>1$.
Throughout the paper we denote
$$   \beta:=\frac{1}{p-1}$$
and
$$ \Gamma_R=\{x\in\R^n,\ 0<x_n<R\}.$$

	\goodbreak
	
	\begin{defi} \rm
	(i) By a solution $u$ of \eqref{eqE1}, we mean a function $u\in C^{2,1}(Q)\cap C(\overline Q)$
	which satisfies the PDE in $Q$ and the boundary conditions pointwise.
	\vskip 2pt
	
(ii) An {\it ancient} (resp., {\it entire}) {\it solution} of \eqref{eqE1} is a 
solution with $I=(-\infty,0)$ (resp., $I=\R$).
		\end{defi}
 
 We stress that this definition does not impose any a priori growth restriction on the solution at infinity.
  We neither require the solution to be $C^1$ in $x$ on the boundary $x_n=0$. 
  Our first main result on problem \eqref{eqE1} completely classifies entire solutions for $p>2$.

 \begin{thm} \label{thmentire}
 Let $p>2$ and let $u$ be an entire solution of \eqref{eqE1}. 
 Then $u$ is stationary and depends only on the variable~$x_n$.
 \end{thm}

 We stress that Theorem~\ref{thmentire} stands in sharp contrast with the case $p=2$.
 Indeed, problem \eqref{eqE1} for $p=2$ admits (positive) 
 entire solutions which are neither stationary nor one-dimensional; see Remark~\ref{rem3}(iii).

\smallskip

\begin{remark} \rm
(i) Theorem~\ref{thmentire} extends \cite{FPS2020} where, as mentioned before,
the conclusion was obtained for the special case of stationary solutions.  
All nontrivial  solutions $u=u(x_n)$ of \eqref{eqE1}  with $p>2$ are explicitly given by
\be{ODEsol}
u(x)=c_p\bigl((x_n+a)^{1-\beta}-a^{1-\beta}\bigr),\quad a\ge 0
\ee
where $1-\beta=\ts\frac{p-2}{p-1}\in(0,1)$ and $c_p=(1-\beta)^{-1}\beta^\beta$. 

 \smallskip

(ii) Applications of Theorem~\ref{thmentire} to the description of blow-up asymptotics 
of solutions of the Cauchy-Dirichlet problem associated with \eqref{eqE1}
will be developed in a forthcoming paper. 
  \end{remark}
 
  It is a natural question whether the classification in Theorem~\ref{thmentire} remains true for ancient solutions.
 Interestingly, it turns out that this is {\it not} the case, as our next result shows the existence of nonstationary ancient solutions.

 \begin{thm}  \label{thmancient2}
 Let $p>1$. There exists a positive, nonstationary ancient solution $u$ of \eqref{eqE1}.
 Namely, $u$ can be found under backward self-similar form 
\be{selfsimilsol}
u(x,t)=|t|^\gamma \phi\biggl(\frac{x_n}{\sqrt{|t|}}\biggr),\quad (x,t)\in \R^n_+\times (-\infty,0),
\ee
where $\gamma=\frac{p-2}{2(p-1)}$ and the profile $\phi\in C^2([0,\infty))$ satisfies $\phi(0)=0$, $\phi'>0$ on $[0,\infty)$ and 
 $\displaystyle\lim_{y\to\infty} \frac{\phi(y)}{y^{\beta+1}}=L>0$.
 \end{thm}

   Nevertheless, we can show that the classification
 in Theorem~\ref{thmentire} remains true for ancient solutions 
 provided a sublinear growth restriction as $x_n\to \infty$ is made on the solution.
 
 \goodbreak
 
  \begin{thm} \label{thmancient}
 Let $p>1$ and let $u$ be an ancient solution of \eqref{eqE1}.
 Assume that $u = o(x_n)$ uniformly for $t$ bounded away from $0$, namely, for each $\eps>0$:
\be{hypthmancient}
\sup_{\Gamma_R\times(-\infty,-\eps]}|u|=o(R),\quad \hbox{as $R\to\infty$}.
\ee
Then $u$ is stationary and depends only on the variable $x_n$. 
 \end{thm}
 
The sublinear growth restriction \eqref{hypthmancient} is in some sense optimal (see Remark~\ref{rem3}(i)).
The above results reveal that the situation for half-space and whole space is notably different: 
For the whole space, the sublinear growth assumption 
\eqref{growthassumpt0} (or \eqref{growthassumpt}) is necessary and sufficient for 
stationarity of both ancient and entire solutions
(cf.~Theorems~\ref{thmLiouvRn0}-\ref{thmLiouvRn} and example \eqref{TW}).
On the contrary, in the half-space case, a growth assumption is necessary for 
stationarity of ancient solutions but not of entire solutions
(cf.~Theorems~\ref{thmentire}-\ref{thmancient}).
Heuristically, this can be interpreted as an effect of the boundary conditions acting as an additional constraint 
or ``barrier'',
which prevents nonstationary ancient solutions to keep existing globally in the future.

Our next result shows positivity 
and  optimal, universal a priori estimates for ancient solutions.

 \begin{thm} \label{thmancient3}
Let $p\ge 2$ and let $u$ be an ancient solution of \eqref{eqE1}.

\begin{itemize}

\item[(i)] Then $u\ge 0$.
 \smallskip 
\item[(ii)] If $p>2$ then we have
 \be{controludsqa}
 u(x,t)\le C(n,p)\Big(x_n^{1-\beta}+x_n^{1+\beta}|t|^{-\beta}\Big)\quad \hbox{in } Q,
 \ee
  and 
  \be{controludsq1a}
 |\nabla u(x,t)|\le C(n,p)\Big(x_n^{-\beta}+x_n^{\beta}|t|^{-\beta}\Big)\quad \hbox{in } Q.
 \ee
 \end{itemize}
 \end{thm}

\smallskip

The assumption $p>2$ in Theorem~\ref{thmancient3}(ii) is sharp,
and estimates \eqref{controludsqa}, \eqref{controludsq1a} are optimal; see Remark~\ref{rem3}(ii).
As for Theorem~\ref{thmancient3}(i), we do not know whether it is still true for $1<p<2$.
However a partial posivity result in that range of $p$, under an additional growth restriction on $u$, is given in 
Proposition~\ref{thmancient3subquadr} below.
The estimates in Theorem~\ref{thmancient3}(ii) will be derived first, and will play an important role in the proof of Theorems~\ref{thmentire} and \ref{thmancient}.

 We close this subsection with some additional remarks.

\begin{remark} \label{rem3}
 \rm
{\bf (Optimality)}
(i) The sublinear growth restriction \eqref{hypthmancient}
 in Theorem~\ref{thmancient} is optimal for $p=2$, in view of the special 
 nonstationary solutions
  \be{entirep2a}
   u(x,t)=\log\bigl[1+e^t\sinh(x_n)\bigr],
   \ee
 which precisely satisfy 
 $\sup_{\,\Gamma_R\times(-\infty,-\eps]} u(x,t)\sim R$ as $R\to\infty$.
  For $p\ne 2$, this restriction is not technical either,
since the nonstationary ancient solutions in Theorem~\ref{thmancient2} grow
 in space like $x_n^{1+\beta}$ as $x_n\to\infty$.
 Although we do not know presently what is the critical growth restriction for fixed $p\ne 2$, we note that 
 it is also nearly optimal for large $p$ since $\beta\to 0$ as $p\to\infty$.
 
     \vskip 1pt

(ii) The optimality of estimates \eqref{controludsqa}, \eqref{controludsq1a} 
in Theorem~\ref{thmancient3}(ii) is seen from the special solutions in \eqref{ODEsol} and in \eqref{selfsimilsol},
 since the latter satisfy the property
\be{selfsimilsol2}
u(x,t)\sim L|t|^{-\beta}x_n^{1+\beta}\ \hbox{as } t\to 0^-,\quad\hbox{for each $x\in\R^n_+$.}
\ee
Moreover, the restriction $p>2$ in Theorem~\ref{thmancient3}(ii) is sharp.
 Actually, estimate \eqref{controludsqa} 
 cannot be true for $p\in (1,2)$ since it would imply that any entire solution satisfies
 $u(x,t)\le C(n,p)x_n^{1-\beta}$
   (applying \eqref{controludsqa} to the solution $v(x,t)=u(x,t+t_0+\tau)$ for each $t_0\in\R$, $\tau>0$,
 taking $t=-\tau$ and letting $\tau\to\infty$). 
 But this estimate is violated by the family of stationary solutions
\be{statsola}
u_a(x_n)=c(p)(a^{1-\beta}-(a+x_n)^{1-\beta}), \quad a>0,
\ee
where $\beta>1$ (cp.~with \eqref{ODEsol}),
  since $\sup_{a>0} u_a(x_n)=\infty$ for all $x_n>0$.
    For $p=2$, 
 \eqref{controludsqa} (resp., \eqref{controludsq1a}) 
would similarly imply that any entire solution satisfies
 $u(x,t)\le C(n,p)$ (resp., $|\nabla u(x, t)|\le C(n,p)x_n^{-1}$),
 but this is violated by the entire solution in \eqref{entirep2a}
  which satisfies $u_{x_n}(x_n,0)\sim 1$ as $x_n\to\infty$.
  For $p=2$ and $n\ge 2$, \eqref{controludsqa} fails even more dramatically since
  the entire solution in \eqref{entirep2} hereafter satisfies
  $\sup_{\,\Gamma_R}u(\cdot,t)=\infty$ for every $t\in\R$ and $R>0$.

     \vskip 1pt
 
(iii) Explicit entire solutions of problem \eqref{eqE1} for $p=2$, which are neither stationary nor one-dimensional
 (for $n\ge 2$),
are given by:
\be{entirep2}
 u(x',x_n,t)=\log\Bigl[1+\exp\bigl(a\cdot x'+(|a|^2+k^2)t\bigr)\sinh(kx_n)\Bigr],
 \ee
 for any $a\in\R^{n-1}\setminus\{0\}$ and $k>0$.
 This problem also has one-dimensional stationary solutions, which are of the form $u=\log(1+kx_n)$.
 Note that these solutions, as well as \eqref{entirep2a}, are obtained by the transformation $u=\log(1+v)$
 where $v$ is a positive solution of the heat equation.
We conjecture that \eqref{eqE1} admits nonstationary entire solutions also for $1<p<2$,
 but this remains an open question.
 \end{remark}

\begin{remark} \rm
 {\bf (Blow-up forms)}
The solution in Theorem~\ref{thmancient2} exhibits global amplitude blow-up on the half-space (cf.~\eqref{selfsimilsol2}).
Its behavior is thus different from the boundary gradient blow-up phenomenon which occurs in the case of bounded domains 
(cf.~the references after \eqref{defgbu}).
Theorem~\ref{thmancient2}  is related to the results in \cite[Section~4.2]{GGK}, 
where backward self-similar blow-up solutions of equation \eqref{eqE0}  on the whole real line  $\Omega=\R$ were constructed.
Additional information on $u$ is given in Proposition \ref{existence} below. 
 \end{remark}

\subsection{Local estimates of Li-Yau and Bernstein type}

 Our last set of main results are the following,
essentially optimal, local estimates for solutions of \eqref{eqE0}
 (see  Proposition~\ref{prop-optim} below for the optimality).
 Beside their intrinsic interest, they are
 important ingredients in the proofs of our Liouville type results.

\vskip 1pt 

 The first one, valid for  
 $p\ge 2$, is a Li-Yau type estimate, which is reminiscent of the celebrated estimate in 
 \cite{LY75} for the heat equation
(which corresponds, here in Euclidean setting, to the case $p=2$ via the usual logarithmic transformation).
 This estimate allows to compare the values of the solution at arbitrary points in space and at different times,
in a way similar to a parabolic Harnack inequality, but in additive instead of multiplicative form.
Such estimates seem completely new for the diffusive Hamilton-Jacobi equation with $p\ne 2$.
  Here and in the rest of the paper, for given $R,T>0$, we denote
$$ B_R:=\{x\in\R^n;\  |x|<R\},\quad 
Q_{R,T}:=B_R \times (0, T),\quad Q'_{R,T}:=B_{R/2} \times (0, T).$$

 \begin{thm}  \label{thmLY}
 Let  $p>2$,  $R, T>0$ and let $u\in C^{2,1}(Q_{R,T})$ be a classical solution of \eqref{eqE0}
 in~$Q_{R,T}$. 
 
\begin{itemize}
\item[(i)]  For each $a \in[0, 1)$, there exists $C = C(n,p,a) > 0$ such that
\be{controlul0}
a|\nabla u|^p-u_t\le  C\bigl(R^{-\beta-1}+R^{1-\beta}t^{-1}\bigr)
\quad\hbox{ in $Q'_{R,T}$.}
\ee

\item[(ii)] There exists $C = C(n,p) > 0$ such that, for all $0<t<s<T$ and $x,y\in B_{R/2}$,
\be{controlul}
 u(x,t)\le u(y,s)+C\Bigl(\frac{|y-x|^p}{s-t}\Bigr)^\beta+C\bigl(R^{-\beta-1}+R^{1-\beta}t^{-1}\bigr)(s-t),
 \ee
 
\end{itemize}
 \end{thm}

     \begin{remark}  \label{remLYp2}
     \rm
In the case $p=2$, \eqref{controlul0} is replaced by:
\be{controlulp2}
a|\nabla u|^2-u_t\le   C(n,a) \big(R^{-2}+ t^{-1}\big),
\ee
 and a similar change for \eqref{controlul}  (cf.~\cite{LY75} and see Step~1 of the proof below).
     \end{remark} 

We stress that, remarkably:
\be{claim-thmLY}
\hbox{For $1<p<2$, estimates \eqref{controlul0} and \eqref{controlul}
fail for every $R>0$ and $a\in[0,1)$,}
 \ee
the heat equation case $p=2$ being thus the endpoint of the validity range;
 see Remark~\ref{remLY}.

Our second  local estimate, this time valid for all $p>1$, is a Bernstein type gradient estimate.
  Let us recall that such estimates go back to the early work \cite{Ber}. The technique
was further developed in, e.g.,~\cite{La58,Ar69,Se69,LY75,Lions85,Bar}
 and \cite{GS1,BVV,BGV}, in pointwise and integral form, respectively.
 For the parabolic diffusive Hamilton-Jacobi equation, Bernstein type estimates have been obtained in
\cite{BL, GilGK, BBL} for the global in space case ($\Omega=\R^n$) and in
\cite{SZ} in the local case.

\begin{thm} \label{propBern}
 Let $p > 1$, $R, T>0$ 
 and let $u\in C^{2,1}(Q_{R,T})$  be a classical solution of \eqref{eqE0} 
 in $Q_{R,T}$ such that
 $M:=\sup_{Q_{R,T}} u<\infty$. We have:
\be{grabound} 
 |\nabla u|\le C(n,p)\left\{\frac{M-u}{R}+\Bigl(\frac{M-u}{R^2\wedge t}\Bigr)^{1/p}\right\}
\ \hbox{ in $Q'_{R,T}$.}
\ee 
 \end{thm}

 We close this subsection by some remarks related with Theorems~\ref{thmLY} and~\ref{propBern}
  and with their optimality.

     \begin{remark}  \rm
    (i) {\bf (Scale invariance)}
We note that estimates \eqref{controlul0} and \eqref{grabound} are invariant by the natural rescaling of the equation:
\be{rescalinglambda}
u_\lambda(x,t):=\lambda^{\beta-1}u(\lambda x,\lambda^2t)\qquad (\lambda>0).
\ee
In particular the full estimates can be recovered from the case $R=1$ by this rescaling.
 
  \smallskip
 
(ii) {\bf (Ancient solutions)}
  Let $p\ge 2$. As a consequence of Theorem~\ref{thmLY}(i), it follows that any ancient solution of
 \eqref{eqE0} in $Q:=\R^n\times(-\infty,0)$
 is subharmonic and nondecreasing in time.
Specifically, $\Delta u=u_t-|\nabla u|^p\ge 0$ in $Q$ (letting $R\to\infty$ and then $a\to 1$
in \eqref{controlul0}).
This cannot be improved in general since the entire solution $u=t+x_n$
satisfies $\Delta u=0$.

  \smallskip
 
(iii) {\bf (Previously known results)}
 Theorem~\ref{propBern} improves the local Bernstein estimate in \cite{SZ} 
which, instead of \eqref{grabound}, took the form
\be{graboundSZ} 
|\nabla u|\le C(n,p) \bigl(R^{-\beta}+R^{-1}+t^{-1/p}\bigr)(M+1-u).
\ee
 It is easy to check that \eqref{grabound} is stronger than \eqref{graboundSZ}
(unlike the latter, \eqref{grabound} for instance quantifies how small $|\nabla u|$ is when $u$ is close to $M$
for given $R, t$).
This improvement is essential in the proofs of Theorems~\ref{thmLiouvRn} and \ref{thmancient}.
    \smallskip
    
  In the global case, it was shown  in \cite{BL,GGK} that the Cauchy problem for \eqref{eqE0} 
  with, e.g., bounded continuous initial data $u_0$ admits a unique classical solution
  such that $|u|\le M:=\|u_0\|_\infty$ and
\be{estBLGGK}
|\nabla u|\le C(n,p)\Big(\frac{M-u}{t}\Big)^{1/p}\quad\hbox{in $\R^n\times(0,\infty)$.}
  \ee
  As a consequence of Theorem~\ref{propBern},  
  we in particular recover \eqref{estBLGGK} as an a priori estimate for any solution of \eqref{eqE0} 
  in $\R^n\times(0,\infty)$ such that $u\le M$ (without reference to possible initial data).
  \end{remark}

  The following proposition shows that estimates \eqref{controlul0} and \eqref{grabound} are essentially optimal,
  and that none of the 
 terms on the right hand side of \eqref{grabound} can be dropped in general.  
    
\begin{prop} \label{prop-optim}
(i) Let $p>2$. There exist $a\in(0,1)$, $C>0$, a family $(u_R)_{R>0}$ of  non-positive 
solutions of  \eqref{eqE0}  in $Q_R=B_R\times (0,\infty)$
and times $t_R>0$, such that 
  \be{ex0}
  \big(a|\nabla u_R|^p-\partial_t u_R\big)(0,t_R)=C\bigl(R^{-\beta-1}+R^{1-\beta}t_R^{-1}\bigr),\quad R>0.
  \ee

(ii) Let $p>1$. There exists a family $(u_\eps)_{\eps\in(0,1)}$ of positive solutions of  \eqref{eqE0}  in $Q=B_2\times (0,1]$
and points $(x_\eps,t_\eps)\in \overline B_1\times (0,1]$ such that, setting  $M_\eps=\sup_{Q} u_\eps$, we have
  \be{ex1} \Big(\frac{M_\eps-u_\eps(x_\eps,t_\eps)}{t_\eps}\Big)^{1/p}\ll |\nabla u_\eps(x_\eps,t_\eps)|=M_\eps-u_\eps(x_\eps,t_\eps)\to\infty,
    \quad\hbox{ as } \eps\to 0,
  \ee
  resp., 
   \be{ex2}
M_\eps-u_\eps(x_\eps,t_\eps)\ll |\nabla u_\eps(x_\eps,t_\eps)|\to\infty,
    \quad\hbox{ as } \eps\to 0.
  \ee
\end{prop}

\subsection{Outline of proofs}
The route map of the proofs is organized in the following order, where each step uses the previous ones:

\smallskip

	\begin{itemize}
	 \setlength{\itemsep}{1pt}
	 	 \setlength{\itemindent}{-14pt}
\item[$\bullet$]Step 1.{\hskip 1mm}Li-Yau  and Bernstein type estimates (Theorems~\ref{thmLY}-\ref{propBern})

\item[$\bullet$]Step 2.{\hskip 1mm}Liouville type theorem for ancient solutions in $\R^n$ (Theorems~\ref{thmLiouvRn0}-\ref{thmLiouvRn})

	\end{itemize}
	
\vskip 1.5pt
\noindent  and then, for the half-space case:
\vskip 1.5pt

	\begin{itemize}
	 \setlength{\itemsep}{1pt}
	 	 \setlength{\itemindent}{-14pt}
		 	 	 \setlength{\parindent}{-20pt}

\item[$\bullet$]Step 3.{\hskip 1mm}A priori estimates   (Theorem~\ref{thmancient3}(ii))

\item[$\bullet$]Step 4.{\hskip 1mm}Existence of forward self-similar solution (Proposition~\ref{positivity1})

\item[$\bullet$]Step 5.{\hskip 1mm}Positivity of any ancient solution (Theorem~\ref{thmancient3}(i))

\item[$\bullet$]Step 6.{\hskip 1mm}Classification of ancient and entire solutions (Theorems~\ref{thmentire} and \ref{thmancient})

\item[$\bullet$]Step 7.{\hskip 1mm}Example of nonstationary ancient (backward self-similar) solution in Theorem~\ref{thmancient2}.

	\end{itemize}
\smallskip
Let us briefly sketch the main ideas of proofs for each step.

\smallskip

 As customary, the proofs of the Li-Yau and Bernstein type estimates \eqref{controlul0}, \eqref{grabound}
 rely on the maximum principle applied to a suitable
auxiliary function of the form $z=\phi\eta$ where $\eta$ is a cut-off. 
For \eqref{controlul0} it turns out that the right choice of $\phi$ is given by 
$\phi=a|\nabla u|^p+|\nabla u|^2-u_t$ with $0<a<1$ and estimate~\eqref{controlul} then follows by integrating  \eqref{controlul0} along a suitable space-time path.
For \eqref{grabound} a rather tricky choice is necessary 
(cf.~\eqref{choiceh}, \eqref{choicef}, \eqref{choicegK}), which is novel as compared with 
previous work (see \cite{SZ}).\footnote{ Related ideas are used in the forthcoming paper \cite{ChSoRiem} to obtain optimal local gradient estimates
for the linear heat equation on noncompact manifolds.}

The proof of the  constancy of ancient solutions in $\R^n$ under  sublinear space-time growth condition (Theorem~\ref{thmLiouvRn})
is a rather direct consequence of our new Bernstein type estimate in Theorem~\ref{propBern},
 but the proof of Theorem~\ref{thmLiouvRn0} (for $p\ge 2$) under the weaker, single-time, space only growth condition \eqref{growthassumpt0} 
requires a more delicate combination with the Li-Yau estimate in Theorem~\ref{thmLY}, suitable integral estimates and a comparison argument.

The proof of the a priori estimates in Theorem~\ref{thmancient3}(ii)  requires several steps.
We first establish some local integral estimates for the positive and negative parts of $u$  (see subsection~\ref{SubSecInt}), obtained by suitable multiplier 
and differential inequality arguments.
We next convert the integral estimates to pointwise ones by making use of our Li-Yau type estimate   \eqref{controlul}.
We then conclude by using a further scaling argument and our Bernstein type estimate.

As for our backward and forward self-similar solutions (Steps 3 and 6),
since we are looking for one-dimensional solutions,
their existence and required asymptotic properties are established by ODE methods,
including shooting arguments and suitable auxiliary functions.

The proof of the positivity of ancient solutions  (Theorem~\ref{thmancient3}(i)) relies, for $p>2$,
 on comparison with a special,
negative forward self-similar solution $\underline u$, with superlinear growth of $\underline u$ at space infinity.
For this, the applicability of the comparison principle crucially relies on the, sublinear in $x_n$, a priori estimate 
in Theorem~\ref{thmancient3}(ii) for the negative part of~$u$.
Since, on the other hand, $\underline u$ decays in time for fixed $x$, this guarantees, via a time shift argument, 
a universal pointwise lower bound for $u$, and we conclude by scaling.

To prove the stationarity and one-dimensionality of ancient solutions in $\R^n_+$ under sublinear growth assumption with respect to~$x_n$ 
(Theorem~\ref{thmancient}),
a key preliminary step is to ensure the decay of $|\nabla u|$ and $u_t$  as $x_n\to \infty$ (uniform with respect to the remaining variables).
Under the sublinear growth condition, this is obtained as a consequence of our new Bernstein type estimate and parabolic regularity. 
With these decay properties at hand, we prove the desired result by a moving planes or translation-compactness argument
(which is a modification of ideas from \cite{FPS2020}, used there for stationary solutions). Since the priori estimates 
in Theorem~\ref{thmancient3}(ii) guarantee that any entire solution satisfies the sublinear growth condition in $x_n$,
the stationarity and one-dimensionality of entire solutions (Theorem~\ref{thmentire}) is then reduced to the case of ancient solutions.

The rest of the paper is organized as follows.
In Section~\ref{SecLY} we prove our Li-Yau and Bernstein type estimates (Theorems~\ref{thmLY}-\ref{propBern} and
 Proposition~\ref{prop-optim}).
Our Liouville type theorems for ancient solutions in~$\R^n$ (Theorems~\ref{thmLiouvRn0}-\ref{thmLiouvRn}) are then derived in Section~\ref{SecLiouvRn}.
In Section~\ref{SecPos} we prove the a priori estimates (Theorem~\ref{thmancient3}(ii)) and the positivity (Theorem~\ref{thmancient3}(i)) of ancient solutions.
Our classification results 
(Theorems~\ref{thmLiouvRn}, \ref{thmentire} and \ref{thmancient}) are then proved in Section~\ref{SecClass}.
 The existence and asymptotic behavior of backward and forward self-similar solutions (Theorem~\ref{thmancient2} and Proposition~\ref{positivity1}) are finally established in Sections~\ref{SecBwd} and \ref{SecFwd}, respectively. 
 Some technical lemmas such as parabolic Kato inequality, blowup and decay differential inequalities are recalled in an appendix.

\section{Proof of Li-Yau and Bernstein type estimates (Theorems~\ref{thmLY}-\ref{propBern})}
 \label{SecLY}

\subsection{Proof of Li-Yau type estimates (Theorem~\ref{thmLY})}

 \indent {\bf Step 1.} {\it Preliminaries.}
It suffices to consider the case $R=1$.
Indeed, the equation being invariant by rescaling, the general case will then follow by applying the result to the rescaled solution
$$v(x,t)=R^{\beta-1}u(Rx,R^2t),\quad x\in B_1,\ t\in(0,R^{-2}T).$$
The proof 
is based on maximum principle applied to the 
auxiliary fonction 
$$\phi:=a|\nabla u|^p+|\nabla u|^2-u_t,$$
multiplied by a cut-off.
 This is a modification of the argument in 
\cite{LY75} for the case $p=2$ 
where the choice $\phi:=a|\nabla u|^2-u_t$ leads to \eqref{controlulp2} in Remark~\ref{remLYp2}.
By interior parabolic regularity, we have $\nabla u, u_t\in C^{2,1}(Q_{R,T})$,
hence $\phi\in C^{2,1}(Q_{R,T})$ (owing to $p>2$).
Set 
$$\begin{cases}
&w:=|\nabla u|^2,\quad m:=p/2\ge 1,\quad
Hu:= u_t-\Delta u, \\
&b:=p|\nabla u|^{p-2}\nabla u,
\quad  \mathcal{L}u:=Hu-b\cdot\nabla u.
\end{cases}$$

{\bf Step 2.} {\it Parabolic inequality for $\phi$.}
We claim that, for some constants $c, C_0>0$ depending only on $n,p,a$, there holds:
\be{1b}
\mathcal{L}\phi\le -c(\phi^2+w^p)
\quad\hbox{in $Q_{R,T}\cap\{\phi\ge C_0\}$.}
\ee

We compute
$$
w_t=2\nabla u\cdot\nabla u_t,\ \ 
 \Delta w=2\nabla u\cdot\nabla \Delta u+2|D^2u|^2, \ \hbox{ where } |D^2u|^2:=\sum_{i,j}(\partial^2_{x_ix_j}u)^2.$$
Using $Hu=|\nabla u|^p=w^m$, we get
 \be{eqHw}
 Hw=2\nabla u\cdot\nabla (Hu)-2|D^2u|^2=2\nabla u\cdot\nabla (w^m)-2|D^2u|^2
 =b\cdot\nabla w-2|D^2u|^2.
 \ee
Next, at points where $|\nabla u|>0$ (or at any point if $m\ge 2$), we have
$\Delta w^m=mw^{m-1}\Delta w+m(m-1)w^{m-2}|\nabla w|^2$.
If $m\in(1,2)$ and $|\nabla u(x_0,t_0)|=0$ then $w^m(x,t_0)=|\nabla u(x,t_0)|^p=O(|x-x_0|^p)$ as $x\to x_0$,
hence $(\Delta w^m)(x_0,t_0)=0$. In all case, we thus have
$$\Delta w^m\ge mw^{m-1}\Delta w,$$
so that, by \eqref{eqHw},
 \be{eqHwm}
 H(w^m)\le mw^{m-1}Hw=b\cdot\nabla (w^m)-pw^{m-1}|D^2u|^2.
 \ee
Since also $\mathcal{L}(u_t)=0$,
it follows from \eqref{eqHw} and \eqref{eqHwm} that $\phi= aw^m+w-u_t$ satisfies
\be{eqLphi}
\mathcal{L}\phi\le -2|D^2u|^2.
\ee
Now observe that, since $m>1$ and $a\in(0,1)$, we have
$$\phi\le \ts\frac{a+1}{2}w^m+\bar C-u_t
=w^m-u_t- \ts\frac{1-a}{2}w^m+\bar C$$
with $\bar C=\bar C(a,p)>0$, hence
\be{eqLphi2}
w^m-u_t\ge\phi-\bar C+ \ts\frac{1-a}{2}w^m\ge \eps_0(\phi+w^m)
\quad\hbox{on $\{\phi\ge 2\bar C\}$,}
\ee
with $\eps_0=\min(\frac12,\frac{1-a}{2})$.
Since, by Cauchy-Schwarz' inequality, 
$$n|D^2u|^2\ge (\Delta u)^2=(w^m-u_t)^2,$$
claim \eqref{1b} with $C_0=2\bar C$ and $c=2n^{-1}\eps_0^2$ follows from \eqref{eqLphi} and \eqref{eqLphi2}.

\smallskip

{\bf Step 3.} {\it Parabolic inequality for $\phi$ with cut-off.}
Let $\gamma\in (0,1)$ to be chosen below 
and $\rho=2/3$. We can select a cut-off function $\eta\in C^2(\bar{B}_1)$,
with 
\be{eta0}
0<\eta\le 1\ \hbox{for $|x|<\rho$},\quad \eta=1\ \hbox{for $|x|\le 1/2$}, 
\quad \eta=0\ \hbox{for $|x|\ge\rho$,}
\ee 
 and
\be{eta}
|\nabla \eta|\le C \eta^{\gamma},\quad |\Delta \eta|+\eta^{-1}|\nabla \eta|^2\le C\eta^{\gamma}
\ee
with $C=C(\gamma)>0$.
 Indeed, it suffices to take $\eta=\theta^k$, 
 where $\theta(x)=\min\big(1,4\big(1-\frac{|x|}{\rho}\big)_+\big)$ 
 and $k\ge 2/(1-\gamma)$.
Set
 $$z=\eta \phi,\qquad Q':=B_{\rho}\times(0,T).$$
We claim that for $\tilde C,A>0$ depending only on $n,p,a$ and $\gamma=(p-1)/p$, there holds
\be{eqz3}
 \mathcal{L}_1z:=Hz+(2\eta^{-1}\nabla \eta-b)\cdot \nabla z
\le -\tilde Cz^2
\quad\hbox{in $Q'\cap\{z\ge A\}$.}
\ee 

In the rest of the proof, all computations take place in $Q'$, unless otherwise specified. 
Using $\mathcal{L}z=\eta \mathcal{L}\phi+\phi \mathcal{L}\eta-2\nabla \eta\cdot\nabla \phi$
and \eqref{1b}, we obtain
\be{eqz1}
\mathcal{L}z\le -c\eta(\phi^2+w^p)-\phi \Delta \eta -\phi b\cdot\nabla \eta -2\nabla \eta\cdot\nabla \phi
\quad\hbox{on $\{\phi\ge C_0\}$.}
\ee
We next control the various terms on the RHS of \eqref{eqz1}:
\begin{enumerate}
\item[(a)] For the second term, by using Young's inequality, we have
$$-\phi \Delta \eta \le \eps\eta \phi^2+\eps^{-1}\eta^{-1}|\Delta\eta |^2.$$
\item[(b)] For the third term, we use the definitions of $b, w$ and Young's inequality to write
\begin{align*}
-\phi b\cdot \nabla \eta
&\le \ts\frac{c}{2}\eta \phi^2+c^{-1}p^2w^{p-1}\eta^{-1}|\nabla \eta|^2\\
&= \ts\frac{c}{2}\eta \phi^2+\big((c\eta)^\frac{p-1}{p}w^{p-1}\big)\big(p^2(c\eta)^{-2+\frac{1}{p}}|\nabla \eta|^2\big)\\
&\le \ts\frac{c}{2}\eta \phi^2+c\eta w^p+\big(p^2(c\eta)^{-2+\frac{1}{p}}|\nabla \eta|^2\big)^p.
\end{align*}
\item[(c)] As for the last term, we control it as follows:
\begin{align*}
-2\nabla \eta\cdot\nabla \phi
&=-2\eta^{-1}\nabla \eta\cdot \eta\nabla \phi\\
&=-2\eta^{-1}\nabla \eta\cdot(\nabla z-\phi \nabla \eta)=-2\eta^{-1}\nabla \eta\cdot\nabla z+2\phi |\nabla \eta|^2\eta^{-1}\\
&\le -2\eta^{-1}\nabla \eta\cdot\nabla z+\eps\eta \phi^{2}+\eps^{-1}\eta^{-3}|\nabla \eta|^4.
\end{align*}
\end{enumerate}
Combining \eqref{eta0}, \eqref{eta}, \eqref{eqz1}, $(a)$-$(c)$, and choosing $\eps=c/8$ and $\gamma=\frac{p-1}{p}\ge \frac12$, we obtain
\be{eqz2}
\mathcal{L}_1z\le -C_1\eta \phi^2+C_2
\quad\hbox{ on $\{z\ge C_0\}$,}
\ee 
 where $C_1=\ts\frac{c}{4}$ and $C_2=C_2(n,p,a)>0$.
This implies claim \eqref{eqz3} with $\tilde C=C_1/2$ and $A= \max\big(C_0,(2C_2/C_1)^{1/2}\big)$.

\smallskip

{\bf Step 4.} {\it Maximum principle argument and completion of proof of \eqref{controlul0}.}
Choose $C_3>0$ such that the function $\psi(t)=C_3t^{-1}$ satisfies $\psi'(t)\ge -\tilde C\psi^2(t)$ for all $t > 0$. 
Fixing $t_0 \in (0,T)$ and setting $\tilde{z}(x,t) = z(x,t+t_0)-\psi(t)-A$,
we have 
$$\mathcal{L}_1\tilde{z}\le 0\quad \hbox{in } \big\{(x,t)\in B_{\rho}\times(0,T-t_0): \tilde{z}\ge 0\big\}.$$
Since $\tilde{z}(x,t)\le 0$ for all sufficiently small $t > 0$, as a consequence of the maximum principle 
(see, e.g., \cite[Proposition~52.4 and Remark~52.11(a)]{QSb19}), we deduce
 that $\tilde{z}\le 0$, i.e.~$z(x,t+t_0)\le C(A+ t^{-1})$ in $B_{\rho}\times(0,T-t_0)$.
Finally, using the definition of $z$ and $\phi$ and letting $t_0\to 0^+$, we obtain the desired result.

\smallskip

{\bf Step 5.} {\it Proof of \eqref{controlul}.}
Following the procedure in \cite{LY75}, we consider the space-time path
 $$\gamma(\tau)=\big(y+\tau(x-y),s-\tau(s-t)\big)\in B_{R/2}\times[t,s],\quad \tau\in[0,1].$$
We have
$$\begin{aligned}
 u(x,t)-u(y,s)
 &=u(\gamma(1))-u(\gamma(0))=\int_0^1\frac{d}{d\tau}\big[u(\gamma(\tau))\big]d\tau\\
& =\int_0^1\Big[(x-y)\cdot\nabla u(\gamma(\tau)) 
 -(s-t)u_t(\gamma(\tau))\Big]d\tau
\end{aligned}$$
so that, using \eqref{controlul0} with $a=\frac12$ and setting $K:=C(n,p)\bigl(R^{-\beta-1}+R^{1-\beta}t^{-1}\bigr)$, we get
$$
 u(x,t)-u(y,s)\le \int_0^1\Big[|x-y||\nabla u|+(s-t)\big(K-\ts\frac12|\nabla u|^p\big)\Big](\gamma(\tau))d\tau.
$$
Setting $g(r)=|x-y|r-\frac{s-t}{2} r^p$, an easy computation gives $\underset{r\ge0} \sup \ g(r) \le C(p)\big(\frac{|x-y|^p}{s-t}\big)^\beta$, 
hence \eqref{controlul}.
\qed

\begin{remark} \rm \label{remLY}
Let us justify claim \eqref{claim-thmLY}.
Assume for contradiction that there exists $R>0$ such that estimate 
\eqref{controlul} (resp., \eqref{controlul0} with some $a\in[0,1)$) 
holds with $C$ independent of $u$.
We may assume $a=0$ without loss of generality.
 By the scale invariance of the equation (cf.~\eqref{rescalinglambda}),
 it is easily seen that it must then be true for every $R>0$.
Since $\beta>1$ for $p\in(1,2)$, letting $R\to\infty$,
it follows that any solution $u$ of  \eqref{eqE0} in $D:=\R^n\times(0,1)$
 must satisfy $u_t\ge 0$ in $D$ (taking $y=x$ in the case of \eqref{controlul}).
 However it is well known (see, e.g.,~\cite{BenA,GilGK}) that for any 
 $u_0\in BC^2(\R^n)$, there exists a solution $u\in C^{2,1}(\R^n\times[0,\infty))$ and,
 taking $u_0$ such that $\nabla u_0(0)=0$ and $\Delta u_0(0)<0$,
 we deduce that $u_t(0,t)=[\Delta u+|\nabla u|^p](0,t)<0$ for $t>0$ small: a contradiction.
\end{remark}

 \subsection{Proof of Bernstein type estimate (Theorem~\ref{propBern})}
Similar to the proof of Theorem~\ref{thmLY}, it suffices to consider the case $R=1$.
The proof is done in three steps.
Whereas Steps 1 and 2 essentially follow the proof of in \cite[Theorem 3.1]{SZ}
(we give the details for completeness),
the main novelty is the key Step 3, which provides the trickier choice of auxiliary function.

\smallskip
{\bf Step 1.} {\it Gradient equation.}
 Assume that $u\in C^{2,1}(Q_{1,T})$
is a classical solution of
$$u_t-\Delta u=|\nabla u|^p \quad\hbox{ in $Q_{1,T}$}.$$
Let $f$ be a function, to be determined later, such that
\be{fmap}
\begin{aligned}
&\hbox{$f\in C^1([0,\infty))\cap C^3(0,\infty)$,\quad $f'>0$ on $(0,\infty)$,}\\
&\hbox{$f$ maps $[0,\infty)$ onto $[-M,\infty)$.}
\end{aligned}
\ee
Let us put
$$v=f^{-1}(-u),\qquad w=|\nabla v|^2.$$
By direct computation, we see that $v$ satisfies the equation
\begin{equation}  \label{eqLocBernsteinPh}
v_t-\Delta v={f''\over f'}(v)\,|\nabla v|^2
-{f'}^{p-1}(v)|\nabla v|^p.
\end{equation}
For $i=1,\dots,n$, put $v_i={\partial v\over\partial x_i}$. 
By parabolic regularity, we have $v_i\in C^{2,1}_{loc}(Q_{1,T})$.
Differentiating (\ref{eqLocBernsteinPh}) in $x_i$, we obtain
\begin{equation}  \label{eqLocBernsteinPhA}
\partial_t v_i-\Delta v_i
=\Bigl[\Bigl({f''\over f'}\Bigl)'\,w-({f'}^{p-1})'\,w^{p\over 2}\Bigl]v_i
+{f''\over f'}\,w_i -{f'}^{p-1}\,(|\nabla v|^p)_i.
\end{equation}
Here and from now on, for brevity, the variable $v$ is omitted
in the expressions $(f''/f')'$, $({f'}^{p-1})'$, etc.
Multiplying (\ref{eqLocBernsteinPhA}) by $2 v_i$, summing up, and using
$\Delta w=2\nabla v\cdot\nabla(\Delta v)+2|D^2v|^2$ (with
$|D^2v|^2=\textstyle\sum_{i,j}(v_{ij})^2$),
we deduce that
\begin{equation}  \label{eqLocBernsteinPhB}
\mathcal{L}w=-2|D^2v|^2+\mathcal{N}w \quad\hbox{ in $Q_{1,T}$},
\end{equation}
where 
$$\mathcal{L}w:= w_t-\Delta w+b\cdot\nabla w,
\quad b:=\Bigl[p{f'}^{p-1}\,w^{p-2\over 2}-2{f''\over f'}\Bigr]\nabla v,$$
\be{defNw}
\mathcal{N}w:=
-2({f'}^{p-1})'\,w^{(p+2)/2}+2\displaystyle\Bigl({f''\over f'}\Bigl)'\,w^2.
\ee

{\bf Step 2.} {\it Cut-off.}
Let $\gamma\in (0,1)$ to be chosen below. 
Fix any $\tau\in(0,T)$ and let $\eta=\eta(x,t):=\eta_1(t)\tilde{\eta}(x)$ 
where $\tilde \eta $ is the cut-off from Step~3 of the proof of Theorem~\ref{thmLY} 
and $\eta_1\in C^1([\tau/4,\tau])$ is a cut-off in time such that:
\be{defeta1}
0\le \eta_1\le 1,\ \ \eta_1(\tau/4)=0, \ \ \eta_1\equiv 1 \ \hbox{in } [\tau/2,\tau),
\ \ |\eta_1'(t)|\le C\tau^{-1}\eta_1^\gamma(t).
\ee
Put 
$$z=\eta w$$
and set $\tilde Q:=B_{3/4}\times(\tau/4,\tau]$.
In the rest of the proof,
all computations will take place in $\tilde Q\cap\{z>0\}$, unless otherwise specified,
and $C$ will denote a generic positive constant depending only on $p,n,\gamma$.
We have
$$
\mathcal{L}z =\eta\mathcal{L}w+w\mathcal{L}\eta-2\nabla\eta\cdot\nabla w
$$
and
$$
\begin{aligned}
2|\nabla\eta\cdot\nabla w|
&=4\eta_1\big|\sum_i\nabla\tilde\eta\cdot v_i\nabla
v_i\big| \leq 4\eta_1\sum_i \tilde\eta^{-1}|\nabla\tilde\eta|^2 v_i^2+\eta_1\sum_i \tilde\eta|\nabla
v_i|^2 \\
&=4\eta_1\tilde\eta^{-1}|\nabla\tilde\eta|^2 w+\eta|D^2v|^2.
\end{aligned}$$
Consequently,
\begin{equation}  \label{eqLocBernsteinPh1}
\mathcal{L}z+\eta|D^2v|^2\leq \eta\, \mathcal{N}w + (\mathcal{L}\eta+4\eta_1\tilde\eta^{-1}|\nabla\tilde\eta|^2)w.
\end{equation}
Since
$|\eta_t|+|\Delta\eta|+\eta_1\tilde\eta^{-1}|\nabla\tilde\eta|^2\le C(1+\tau^{-1})\eta^\gamma$
owing to \eqref{eta}, \eqref{defeta1}, and
$$
|(b\cdot\nabla\eta)w|
\leq C\eta^\gamma\Bigl|p{f'}^{p-1}\,w^{(p+1)/2}-2{f''\over f'}w^{3/2}\Bigr|,
$$
we deduce from \eqref{eqLocBernsteinPh1} and property of $\eta$ that
\begin{equation}  \label{eqLocBernsteinPh2}
\mathcal{L}z+\eta|D^2v|^2\leq \eta\, \mathcal{N}w +
C\eta^\gamma(\tau^{-1}+1) w+ C\eta^\gamma\Bigl|p{f'}^{p-1}\,w^{(p+1)/2}-2{f''\over
f'}w^{3/2}\Bigr|.
\end{equation}

{\bf Step 3.} {\it Choice of auxiliary function and completion of proof of Theorem~\ref{propBern}.}
Let $g\in C([0,\infty))\cap  C^2(0,\infty)$, to be determined below, satisfy
\be{hypg1}
g,g'>0\ \hbox{and  $g''\le 0$ in $(0,\infty)$,} 
\ee
\be{hypg2}
\int_0^1\frac{d\tau}{g(\tau)}<\infty,\quad \int_1^\infty\frac{d\tau}{g(\tau)}=\infty.
\ee
We set
\be{choiceh}
h(s)=\int_0^s\frac{d\tau}{g(\tau)}
\ee
and then
\be{choicef}
f(v)=h^{-1}(v)-M\quad \hbox{i.e.}\quad v=h(M-u),
\ee
and note that $v\geq 0$ in $Q_{1,T}$ due to $u\leq M$ and $g>0$ in $[0,\infty)$. 
With this
choice, we see that \eqref{fmap} is satisfied, and we have
\begin{align*}
f'&=g\circ h^{-1},\qquad \quad
\big({f'}^{p-1}\big)'=(p-1)\big(g'g^{p-1}\big)\circ h^{-1},\\
\frac{f''}{f'}&=g'\circ h^{-1},\qquad \quad
\left(\frac{f''}{f'}\right)'= \big(gg''\big)\circ h^{-1}.
\end{align*}
Formula \eqref{eqLocBernsteinPh2} then implies
\be{eqLzN}
\mathcal{L}z\leq \eta\, \mathcal{N}w +
\underbrace{C\eta^\gamma (\tau^{-1}+1) w}_{T_1}+\underbrace{C\eta^\gamma g^{p-1}(h^{-1}(v)) w^{(p+1)/2}}_{T_2}+\underbrace{C\eta^\gamma g'(h^{-1}(v))w^{3/2}}_{T3},
\ee
whereas
$$
\mathcal{N}w=\underbrace{-2(p-1)\big(g'g^{p-1}\big)(h^{-1}(v))w^{(p+2)/2}}_{T_4}+\underbrace{2 \big(gg''\big)(h^{-1}(v))w^2}_{T_5}.
$$
Recall that $g'>0$ and $g''\le 0$.
We shall now split and gather the terms on the right hand side of \eqref{eqLzN}
so as to control each of the positive terms $T_1, T_2, T_3$ by the negative terms $T_4, T_5$.
Omitting without risk of confusion the variables $h^{-1}(v)$ in the terms involving $g$
and since we restrict to the set $\{z>0\}$ (where $0<\eta\le 1$), we may write
\begin{align*}
\mathcal{L}z
&\leq  -2(p-1)\eta  g'g^{p-1} w^{\frac{p+2}{2}}+2 \eta gg''w^2\\
&\qquad\qquad+C\eta^\gamma (\tau^{-1}+1) w+C\eta^\gamma g^{p-1} w^{\frac{p+1}{2}}+C\eta^\gamma g'w^{\frac32}\\
&=  -2(p-1)g'g^{p-1} \eta^{-\frac{p}{2}} z^{\frac{p+2}{2}}+2 gg''\eta^{-1}z^2\\
&\qquad\qquad+C\eta^{\gamma-1} (\tau^{-1}+1) z
+C\eta^{\gamma-\frac{p+1}{2}} g^{p-1} z^{{\frac{p+1}{2}}}+C\eta^{\gamma-\frac32}g'z^{\frac32}\\
&=\eta^{\gamma-1}z \tilde T_1+\eta^{\gamma-\frac{p+1}{2}} g^{p-1}z^{\frac{p+1}{2}} \tilde T_2+\eta^{\gamma-\frac32}z^{\frac32} \tilde T_3,\end{align*}
where
$$\begin{aligned}
\tilde T_1&=C (\tau^{-1}+1)-(p-1)g'g^{p-1} \eta^{1-\gamma-\frac{p}{2}} z^{p/2}+gg''\eta^{-\gamma}z\\
\tilde T_2&=  C  -\textstyle\frac{p-1}{2}g' \eta^{\frac12-\gamma}z^{1/2}\\
\tilde T_3&=Cg' -\textstyle\frac{p-1}{2}g'g^{p-1} \eta^{\frac{3-p}{2}-\gamma} z^{(p-1)/2}+gg''\eta^{\frac12-\gamma}z^{1/2}.
\end{aligned}$$
Taking $\gamma=\gamma(p)\in\big(\max(\frac12,\frac{3-p}{2}),1\big)$ and using \eqref{hypg1}, we obtain
\be{ineqT123}
\left.
\begin{aligned}
\tilde T_1&\le C (\tau^{-1}+1)-(p-1)g'g^{p-1} z^{p/2}+gg''z\\
\tilde T_2&\le  C -\textstyle\frac{p-1}{2}g' z^{1/2} \\
\tilde T_3&\le Cg' -\textstyle\frac{p-1}{2}g'g^{p-1} z^{(p-1)/2}+gg''z^{1/2}.
\end{aligned}
\quad\right\}
\ee
In view of ensuring the negativity of $\tilde T_i$ at points where $z$ is large, we thus look for a function $g$ satisfying
\be{choiceg0}
g'\ge 1,\quad g'g^{p-1} - gg'' \ge c(p)(g'+\tau^{-1})
\ee
for some $c(p)>0$.
Choosing 
\be{choicegK}
g(s)=s+Ks^\theta,\quad \hbox{with $\theta=1/p$ and $K=(1+\tau^{-\theta})$},
\ee
we obtain
$$\begin{aligned}
g'g^{p-1}-gg''
&=\big(1+K\theta s^{\theta-1}\big)\big(s+Ks^\theta\big)^{p-1}+\theta(1-\theta)Ks^{\theta-2}\big(s+Ks^\theta\big)\\
&\ge \theta s^{p-1} \big(1+K^p s^{p(\theta-1)}\big)+\theta(1-\theta)\big(Ks^{\theta-1}+(Ks^{\theta-1})^2\big)\\
&\ge \theta(1-\theta) (s^{p-1}+K^p+Ks^{\theta-1}+(Ks^{\theta-1})^2)\\
&\ge \ts\frac{\theta(1-\theta)}{2} (g'(s)+\tau^{-1})
\end{aligned}$$
hence \eqref{choiceg0}, and
we see that \eqref{hypg1}, \eqref{hypg2} are also satisfied. 
It follows from \eqref{ineqT123}, \eqref{choiceg0} that there exists $A=A(n,p)>0$ such that $\tilde{T}_i\le 0$ 
whenever $z\ge A$. It then follows that 
$$
\mathcal{L} z\leq 0
\quad\hbox{ in $\{(x,t)\in \tilde Q;\ z(x,t)\geq A\}$}.
$$
Since $z=0$ on the parabolic boundary of $\tilde Q$,
 the maximum principle yields $z\le A$ in $\tilde Q$.
Finally, using $z=\eta|\nabla v|^2$, \eqref{choiceh}, \eqref{choicef} and the definition of $\eta$, this implies
$$|\nabla u|=|\nabla v|\, g(M-u)\le A^{1/2} g(M-u)
\le C\Big\{M-u+(1+\tau^{-\frac1p})\left(M-u\right)^{\frac1p}\Big\}$$ 
in $B_{1/2}\times[\tau/2,\tau]$.
For each $t\in (0,T)$, applying this 
with $\tau=t$ yields 
$$
 |\nabla u|\le C(n,p)\Big\{M-u+(1+t^{-\frac1p})\left(M-u\right)^{\frac1p}\Big\}
\ \hbox{ in $B_{1/2}\times(0,T]$,}
$$
 which is equivalent to \eqref{grabound} for $R=1$.
\qed

 \subsection{Proof of Proposition~\ref{prop-optim}}
(i) We consider the forward self-similar solutions 
 \be{propphi1}
 v(x,t)=t^\gamma \phi(x_n/\sqrt t),\quad\gamma=(1-\beta)/2,
 \ee
 of the equation $v_t-\Delta v+|\nabla v|^p=0$ in the half-space,
  given for $p>2$ by 
 Lemmas~\ref{lemJ1}-\ref{lemJ2} below with $\alpha\in J_1$.
 We know that there exists $\lambda>1$ such that
 \be{propphi2}
\phi''(y)<0<\phi'(y),\quad y\ge \lambda.
\ee
  For each $R>0$, we then set  $u_R(x,t)=-v(x+\lambda R e_n,t)$, defined in 
 $Q_R$. For $t_R=R^2$, it satisfies
 $$\big(a|\nabla u_R|^p-\partial_t u_R\big)(0,t_R)=\big((a-1)|\nabla u_R|^p-\Delta u_R\big)(0,t_R)
=L(a,\lambda) R^{-\beta-1}$$
where $L(a,\lambda)=(a-1)|\phi'(\lambda)|^p-\phi''(\lambda)$ in view of \eqref{propphi1}.
Taking $a\in(0,1)$ close to $1$, we have $L(a,\lambda)>0$ by \eqref{propphi2}, hence \eqref{ex0}.
 
  \smallskip

  (ii) Recall $Q=B_2\times (0,1]$.

\smallskip

$\bullet$ Proof of \eqref{ex1}.  Consider the solutions 
    \be{ExOptim2a}
    u_\eps(x,t)=\eps^{-1}x_1+\eps^{-p}t,\quad (x,t)\in Q.
    \ee
    Setting
$t_\eps=1$, $x_\eps=e_1$,
we have $|\nabla u_\eps(e_1,1)|=\eps^{-1}=M_\eps-u_\eps(e_1,1)\gg(\frac{M_\eps-u_\eps(e_1,1)}{t_\eps})^{1/p}$,
hence \eqref{ex1}.

\smallskip

$\bullet$ Proof of \eqref{ex2}. First assume $p>2$. We may assume $n=1$ (then setting $u_\eps(x,t)=u_\eps(x_1,t)$).
Let 
\be{fwdform}
 U(x,t)=t^\gamma \phi(x/\sqrt t),\quad x\in\R^n,\ t>0,\qquad \gamma=(1-\beta)/2,
 \ee 
be the forward self-similar solution from \cite[Section~4]{GK} with $p>2$ and $n=1$, which satisfies
 \be{estimUGGK}
 0<U(x,t)\le C(|x|+\sqrt{t})^{1-\beta}.
 \ee
 Set 
 \be{defueps}
 u_\eps(x,t)=U(x,t+\eps),\quad t_\eps=\eps,\quad |x_\eps|=\sqrt\eps.
 \ee
 Since $\beta<1$, we have $M_\eps\le C$ by \eqref{estimUGGK} and $|\nabla u_\eps(x_\eps,t_\eps)|=C\eps^{-\beta/2}$, 
 hence \eqref{ex2}.
 
Next consider the case $p\in (1,2)$.
Instead of $U$, we use a forward self-similar solution
from \cite[Theorem~2.1]{QW} (after changing $u$ in $-u$, since the result is stated for the equation with nonlinearity $-|\nabla u|^p$).
It takes the form \eqref{fwdform}, where the profile satisfies $\phi=\phi(r)$, with $r=|x|$, 
 $\phi,\phi'\in L^\infty(0,\infty)$ and $\phi<0<\phi'$ in $(0,\infty)$. Keeping the notation \eqref{defueps}, it follows that
 $|u_\eps(x_\eps,t_\eps)|\le C\eps^{(1-\beta)/2}$ and $|\nabla u_\eps(x_\eps,t_\eps)|=C_1\eps^{-\beta/2}$, 
so that we reach the same conclusion.

Finally, for $p=2$, it suffices to replace $U$ with the solution $U(x,t)=-\frac{n}{2}\log t-\frac{|x|^2}{4t}$
(obtained as the $\log$ of the heat kernel),
which now yields $M_\eps-u_\eps(x_\eps,t_\eps)=\frac{n}{2}\log 2+\frac{1}{16}=c$ and $|\nabla u_\eps(x_\eps,t_\eps)|=C_1\eps^{-1/2}$,
 hence again~\eqref{ex2}.
\qed

\section{Proof of Liouville type theorem for ancient solutions in~$\R^n$ (Theorems~\ref{thmLiouvRn0}-\ref{thmLiouvRn})}
 \label{SecLiouvRn}

As a consequence of the Bernstein estimate in Theorem~\ref{propBern}, we first have the following lemma,
 which in particular contains Theorem~\ref{thmLiouvRn}.

  \begin{lem}\label{lemLiouvRn}
Let $p>1$ and let $u\in C^{2,1}(Q)$ be a solution of \eqref{eqE0} in $Q:={\R^n}\times (-\infty,0)$
satisfying \eqref{growthassumpt}.
Then $u$ is constant. 
\end{lem}

\begin{proof} Fix $x_0\in\R^n$ and $t_0\in (-\infty,0)$.
Let $R\geq 1$, $Q=B_R\times (0,R)$,
and consider the function $U(x,t) = u(x+x_0,t+t_0-R)$.
We have $U\leq M_R$ in $\overline Q$, where
$$M_R:=\sup_{B(x_0,R)\times
(t_0-R,t_0)}u\le \eps(R)R,\quad\hbox{ with }\lim_{R\to \infty} \eps(R)=0,$$
owing to (\ref{growthassumpt}).
Applying Theorem~\ref{propBern} to the function $U$ in $Q$, it follows that, {for $R\ge 1$,}
\begin{align*}
|\nabla u(x_0,t_0)|=|\nabla U(0,R)|
&\le C\big(\ts\frac{M_R-U(0,R)}{R}+\big(\frac{M_R-U(0,R)}{R^2}\big)^{\frac1p}
+\big(\frac{M_R-U(0,R)}{R}\big)^{\frac1p}\big)\\
&\le C\big(\eps(R)-\ts\frac{u(x_0,t_0)}{R}+\big(\eps(R)-\frac{u(x_0,t_0)}{R}\big)^{\frac1p}\big) 
\end{align*}
and the result follows by letting $R\to\infty$ in the last inequality. 
\end{proof}

Starting from Lemma~\ref{lemLiouvRn}, the proof of Theorem~\ref{thmLiouvRn0}, under the weaker assumption \eqref{growthassumpt0}, 
is divided in three steps.

\begin{proof}[Proof of Theorem~\ref{thmLiouvRn0}]

{\bf Step 1.} {\it Constancy in the past.} 
We claim that there exist a constant $B\in \R$ and a maximal $t_1\in (-\infty,0]$ such that
\be{uequivB}
u\equiv B\quad\hbox{ in $\R^n\times(-\infty,t_1)$.}
\ee

Letting $R\to\infty$ and then $a\to 1$
in the Li-Yau type estimate \eqref{controlul0}, we have
\be{LY2}
u_t-|\nabla u|^p=\Delta u\ge 0\quad\hbox{ in $Q$.}
\ee
In particular
$$\sup_{t\le t_0} u(x,t)\le u(x,t_0),\quad x\in \R^n$$
so that assumption \eqref{growthassumpt0} imples \eqref{growthassumpt}, 
and Lemma~\ref{lemLiouvRn} guarantees \eqref{uequivB} with $t_1$ replaced by $t_0$.
It thus suffices to 
let $t_1:=\max \{\tau<0;\ u\equiv B \hbox{ in $\R^n\times(-\infty,\tau]$}\}\in [t_0,0]$.

\smallskip

{\bf Step 2.} {\it A priori estimates for $t>t_1$.} 
We may assume $B=0$ without loss of generality.
Assume for contradiction that $t_1<0$.
Then $u(\cdot,t_1)\equiv 0$ by \eqref{uequivB} and continuity.

To reach a contradiction we want to show that $u$ actually has to be identically $0$ for $t>t_1$ close to $t_1$.
This is nontrivial, since uniqueness for the Cauchy problem without growth assumption on the solution 
cannot be taken for granted. Our next task is thus to obtain suitable a priori estimates for $t>t_1$.
We claim that there exists a constant $C_0>0$ (depending on $u$)
such that
\be{AEm}
0\le u\le C_0(1+|x|)^{\beta+1}, \quad |\nabla u|\le C_0(1+|x|)^\beta\ \ \hbox{ in $\R^n\times [t_1,t_1/2]$.}
\ee
The inequality $0\le u$ in \eqref{AEm} immediately follows from \eqref{LY2}.

Let us show the upper bound of $u$ in \eqref{AEm}.
First, since $u$ is a classical solution, there exists a constant $M>0$ such that
\be{bounduB1}
u\le M \ \ \hbox{ in $\overline{B_1}\times[t_1,t_1/4]$.}
\ee
We shall write $u(x,t)=u(r,\omega,t)$ where $(r,\omega)$ are the hyperspherical coordinates.
Fix $t\in [t_1,t_1/4]$, $R>1$ and any direction $\omega\in S^{n-1}$.
 By H\"older's inequality, we have, for all $r\in(1,R)$,
 $$\begin{aligned} 
 u^p(r,\omega,t)
&\le \Big(M+ \int_1^r |u_r(\rho,\omega,t)| \, d\rho\Big)^p\\
&\le 2^{p-1}M^p+ 2^{p-1}\Big(\int_1^R |u_r(\rho,\omega,t)| \, d\rho\Big)^p\\
&\le 2^{p-1}M^p+ (2R)^{p-1}\int_1^R |u_r(\rho,\omega,t)|^p \, d\rho,
 \end{aligned}$$
 and this remains true for $r\in (0,1]$ owing to  \eqref{bounduB1}. Applying H\"older's inequality once again we get
  $$\begin{aligned} 
\Big(\int_0^R u(r,\omega,t)\, dr\Big)^p
 &\le R^{p-1} \int_0^R u^p(r,\omega,t)\, dr\\
&\le CR^pM^p+ CR^{2p-1}\int_0^R |u_r(r,\omega,t)|^p \, dr.
 \end{aligned}$$
 Here and hereafter $C, c$ denote generic positive constants independent of $R, t, \omega$ (but possibly depending on $u$ and $t_1$).
Next integrating the inequality $u_t\ge |\nabla u|^p$ from \eqref{LY2} along the line segment of endpoints $0$ and $R\omega$, we deduce that,
for all $t\in(t_1,t_1/4)$,
  $$\begin{aligned} 
  \frac{d}{dt} \int_0^R u(r,\omega,t)\, dr
  &\ge  \int_0^R |(\nabla u)(r,\omega,t)|^p\, dr\\
&\ge cR^{1-2p}\Big(\int_0^R u(r,\omega,t)\, dr\Big)^p-CR^{1-p}.
 \end{aligned}$$
We deduce from Lemma~\ref{blowupineq}(i) that, for all $t\in(t_1,t_1/2]$,
  $$\int_0^R u(r,\omega,t)\, dr\le C\big[R^{\frac{2p-1}{p-1}}+R\big]\le C R^{\beta+2}.$$
Consequently,
    $$\int_0^R u(r,\omega,t)r^{n-1}\, dr\le R^{n-1}\int_0^R u(r,\omega,t)\, dr\le C R^{n+\beta+1},$$
hence
 $$\int_{B_R}u(t)\,dx=\int_{S^{n-1}}\int_0^R u(r,\omega,t)r^{n-1}\, drd\omega\le C R^{n+\beta+1},\quad t_1\le t\le t_1/2.$$

  Now, since $u(\cdot,t)$ is subharmonic by \eqref{LY2}, it follows from the mean value inequality that, 
  for each $x$ with $|x|>1$,
$$u(x,t)\le \fint_{B_{|x|}(x)}u\,dx\le C|x|^{-n}\int_{B_{2|x|}}u\,dx\le C |x|^{\beta+1},\quad t_1\le t\le t_1/2.$$
This combined with \eqref{bounduB1} guarantees the upper bound on $u$ in \eqref{AEm}.
Finally, since the latter bound is actually true in $\R^n\times(-\infty,t_1/2)$, 
combining it with the Bernstein estimate in Theorem~\ref{propBern}, we obtain, after a time shift,
$$|\nabla u|\le C(n,p)\left(\frac{R^{\beta+1}-u}{R}+\Bigl(\frac{R^{\beta+1}-u}{R^2}\Bigr)^{1/p}\right)\ \ \hbox{in $B_R\times(-\infty,t_1/2)$},$$
which implies the gradient bound in \eqref{AEm}.
\smallskip

{\bf Step 3.} {\it Comparison argument and conclusion.} 
We use a modification of a device from \cite[Section~3]{GGK}, where it is used for a different purpose.
Fix any $\eps>0$ and set
$$v_\eps(x,t)=\eps e^{2Lt}(1+|x|^2)^k$$
where $k=\beta+2$ and $L>0$ is to be chosen.
By direct computation, we have 
$$\partial_t v_\eps=2Lv_\eps,\quad \nabla v_\eps=\frac{2kx}{1+x^2}v_\eps,\quad \Delta v_\eps\le \lambda v_\eps.$$
for some $\lambda=\lambda(n,k)>0$. 
Taking $L>\lambda$, we get
$$\mathcal{L}v_\eps:=\partial_t v_\eps-\Delta v_\eps-(2k)^{-1}L(1+|x|)|\nabla v_\eps|\ge (2L-\lambda-L)v_\eps>0.$$
On the other hand, taking $L>\max(2kC_0^{p-1},\lambda)$ and using the second part of \eqref{AEm}, we have, in $\R^n\times [t_1,t_1/2]$,
  $$\begin{aligned} 
  \mathcal{L}u
  &=|\nabla u|^p-(2k)^{-1}L(1+|x|)|\nabla u|=\big(|\nabla u|^{p-1}-(2k)^{-1}L(1+|x|)\big)|\nabla u|\\
    &\le \big(C_0^{p-1}-(2k)^{-1}L\big)(1+|x|)|\nabla u|<0.
     \end{aligned}$$
     In addition, the first part of \eqref{AEm} guarantees that, for any $\eps>0$, 
     there exists $R_\eps>1$ such that $u<v_\eps$ for $|x|\ge R_\eps$ and $t\in(t_1,t_1/2]$.
     Since $u(\cdot,t_1)\equiv 0$, it follows from the maximum principle that $u\le v_\eps$ in $\R^n\times(t_1,t_1/2]$.
     Letting $\eps\to 0$, we deduce that $u\equiv 0$ in $\R^n\times(-\infty,t_1/2]$,
     which contradicts the maximality of $t_1$ in \eqref{uequivB}.
     Consequently $t_1=0$ and the theorem is proved.
\end{proof}

\section{Proof of positivity and a priori estimates (Theorem~\ref{thmancient3})}
 \label{SecPos}
 
\subsection{Integral estimates for $p>2$}
 \label{SubSecInt}
 
We here provide preliminary integral estimates for the positive and negative parts of $u$.
Since $p>2$ we may fix
$$\alpha\in (1,p-1).$$
In this subsection $C, C_i$ denote generic positive constants depending only on $n,p,\alpha$,
and we set $s_+=\max(s,0)$, $s_-=\max(-s,0)$.
 
 \begin{prop}\label{1L3}
  Let $p>2$ and let $u$ be an ancient solution of \eqref{eqE1}. 
 For all $a\in\R^{n-1}$ and $t<0$, we have 
 \be{1integralcontrol}
\int_{B_2'(a)\times(0,2)}u_+(x,t) \,x_n^\alpha\,dx\le C\big(1+|t|^{-\beta}\big),
  \ee
  and
 \be{vap3}
 \int_{B_2'(a)\times(0,2)} u_-(x,t)  \,x_n^\alpha \,dx\le C.\ee
 \end{prop}

  In view of its proof we may assume $a=0$ without loss of generality and we define the cylinders
$$\Omega=B'_3\times(0,3),\qquad
\Omega_\eps=\Omega\cap\{x_n>\eps\},\quad \eps>0.$$
Since $u$ is assumed only continuous (and not $C^1$) up to $x_n=0$, we will argue on $\Omega_\eps$.
For the sake of passing to the limit $\eps\to 0$, we will make use of the Bernstein estimate in \eqref{grabound}. 
Namely, since $p>2$, for given $t_0<t_1<0$, the latter guarantees that
\be{BernSZ}
|\nabla u(x,t)|\le C(M_0+1)x_n^{-1},\quad x\in B'_3\times(0,3),\ t_0<t<t_1, 
\ee 
where $M_0:=\underset{\Gamma_4\times(t_0-1,t_1)}\sup |u|<\infty.$ 
We also fix a cylindrical cut-off function $\varphi$ such that
\be{defvarphi}
\begin{aligned}
&\varphi(x):=\psi(x')\chi(x_n),\quad\hbox{where $\chi (s)=(3-s)s^\alpha$ and}\\
 &\psi\in C^2(B_3'),\ 0\le\psi\le 1,\ |\nabla \psi|\le \psi^{1/p},
 \ \psi\equiv 1\mbox{ in $B_2'$},\ \psi=0 \mbox{ on $\partial B'_3$}.
 \end{aligned}
 \ee
With this cut-off,  we shall use the following properties, 
 including the weighted Poincar\'e type inequality \eqref{euphi2}.
 
 \begin{lem}\label{genlem}
For all $w\in W^{1,p}(\Omega_\eps)\cap C(\overline{\Omega_\eps})$, we have
\be{euphi2a}
\Big|\int_{\Omega_\eps}\nabla w\nabla \vap \,dx\Big|\le \frac12 \int_{\Omega_\eps}|\nabla w|^p\vap \,dx +C
\ee
and
\be{euphi2}
\int_{\Omega_\eps} |w|^p \psi(x') \,dx\le C\Bigg(\int_{\Omega_\eps}|\nabla w|^p\vap \,dx+\int_{B'_3} |w(x',\eps)|^p \psi(x')\,dx'\Bigg).
\ee
\end{lem}

\begin{proof}[Proof of Lemma~\ref{genlem}]
(i) By Young's inequality we have 
$$
\begin{aligned}
\Big|\int_{\Omega_\eps}\nabla w\nabla \vap \,dx\Big|
&= \Big|\int_{\Omega_\eps} (\vap^{\frac1p}\nabla w)\cdot(\vap^{-\frac1p}\nabla \vap) \,dx\Big|\\
&\le \frac12 \int_{\Omega_\eps}|\nabla w|^p\vap \,dx +C\int_{\Omega_\eps}\Big(\vap^{-1}|\nabla \vap|^p\Big)^\beta \,dx=:J_1+J_2.
\end{aligned}
$$
To control $J_2$, we note that
$$\vap^{-1}|\nabla \vap|^p\le \chi^{p-1}\psi^{-1}|\nabla \psi|^p+\psi^{p-1}\chi^{-1}|\chi'|^p\le C+\psi^{p-1}\chi^{-1}|\chi'|^p$$
and, since $|\chi'|\le Cx_n^{\alpha-1}$, we get
$$
\big(\chi^{-1}|\chi'|^p\big)^\beta
\le C\big[x_n^{(\alpha-1)p-\alpha}(3-x_n)^{-1}\big]^\beta
=Cx_n^{\alpha-\beta-1}(3-x_n)^{-\beta}\in L^1(0,3)
$$
owing to $\alpha>\beta$ and $\beta\in(0,1)$. 
Consequently, $J_2\le C$ (independent of $\eps$), hence \eqref{euphi2a}.
\smallskip

(ii) By H\"{o}lder's inequality, we have 
$$\begin{aligned}
\int_\eps^3|\nabla w(x',s)|ds
&=\int_\eps^3 (|\nabla w(x',s)|\chi^{1/p})\chi^{-1/p}ds\\
&\le \Big(\int_\eps^3|\nabla w(x',s)|^p\chi(s)ds\Big)^{\frac1p}\Big(\int_\eps^3\chi^{-\beta} (s)ds\Big)^{\frac{p-1}{p}}\\
&\le C\Big(\int_\eps^3|\nabla w(x',s)|^p\chi(s)ds\Big)^{\frac1p}
\end{aligned}$$
where we used $\chi^{-\beta}=s^{-\alpha\beta}(3-s)^{-\beta}\in L^1(0,3)$ owing to $\alpha<p-1$ and $p>2$.
It follows that
$$\begin{aligned}
|w(x',x_n)|^p
 &=\Big|w(x',\eps)+\int_\eps^{x_n} w_{x_n}(x',s)ds\Big|^p\\
 &\le 2^{p-1}|w(x',\eps)|^p+C\int_\eps^3|\nabla w(x',s)|^p\chi(s)ds.
\end{aligned}$$
Multiplying by $\psi(x')$ and integrating over $B'_3$, we obtain
 $$\int_{B'_3}|w(x',x_n)|^p\psi(x')\,dx'
\le C\int_{B'_3}|w(x',\eps)|^p\psi(x')\,dx'+C\int_{\Omega_\eps}|\nabla w|^p\vap \,dx.$$
Finally, integrating in $x_n$ over $(\eps,3)$ 
 implies \eqref{euphi2}. 
\end{proof}

\begin{proof}[Proof of Proposition \ref{1L3}]
{\bf Step 1.} {\it Upper bound for $\int_\Omega u(t)\vap \,dx$.}
We claim that
\be{upint}
\int_\Omega u(t)\vap \,dx\le C\big(1+|t|^{-\beta}\big),\quad t<0.
\ee

Let $\eps\in(0,1)$ and set $y_\eps(t)=\int_{\Omega_\eps} u(t)\vap \,dx$. Multiplying the PDE in \eqref{eqE1} by $\varphi$, integrating by parts over $\Omega_\eps$
and using $\varphi=0$ on $\{x_n>\eps\}\cap\partial\Omega_\eps$, we get
$$\begin{aligned}
y'_\eps(t)
&= \int_{\Omega_\eps}|\nabla u|^p\vap \,dx+\int_{\Omega_\eps} \vap\Delta u \,dx\\
& =\int_{\Omega_\eps} |\nabla u|^p\vap \,dx-\int_{\Omega_\eps}\nabla u\nabla \vap \,dx+\int_{\{x_n=\eps\}}  \vap u_\nu \,dx'.
 \end{aligned}$$
 Next set 
 \be{defIJeps}
 I_\eps(t):=\int_{B'_3} \vap |u_\nu(x',\eps,t)| \,dx',\quad J_\eps(t):=C\int_{B'_3} |u(x',\eps,t)|^p \psi(x')\,dx'.
 \ee
 Using \eqref{euphi2a},  \eqref{euphi2}, $\varphi\le C\psi$ and H\"older's inequality, we obtain
$$\begin{aligned}
y'_\eps(t)\
&\ge \frac12 \int_{\Omega_\eps} |\nabla u|^p\vap \,dx-I_\eps -C\\
 &\ge C_1\int_{\Omega_\eps} |u|^p \vap \,dx-I_\eps-J_\eps-C
  \ge C_1 |y_\eps|^p -I_\eps-J_\eps-C.
 \end{aligned}$$
Fix $t_0<t_1<0$. Using \eqref{BernSZ} and $0\le \varphi\le Cx_n^\alpha$, we deduce that
$|I_\eps| \le C(M_0+1) \eps^{\alpha-1}$ for all $t\in(t_0,t_1)$.
This along with the assumption $u\in C(\overline{\R^n_+}\times(-\infty,0))$ and the boundary conditions
guarantees the existence of $\eps_0\in(0,1)$ (depending on $u,t_0,t_1$) such that 
\be{boundIeps}
I_\eps(t)+J_\eps(t)\le 1\quad\hbox{ for all $t\in(t_0,t_1)$ and $\eps\in(0,\eps_0)$.}
\ee
Consequently,
$$y'_\eps(t)\ge  C_1 |y_\eps|^p -C,\quad t_0<t<t_1,\ \eps\in(0,\eps_0).$$
It follows (cf.~Lemma~\ref{blowupineq}(i)) that
 $$y_\eps(t)\le C_2\big[1+(t_1-t)^{-\beta}\big],\quad t_0<t<t_1,\ \eps\in(0,\eps_0).$$
 Letting $\eps\to 0$ and then $t_1\to 0$, $t_0\to-\infty$, this implies claim \eqref{upint}.

\smallskip

 {\bf Step 2.} {\it Bound for $\int_\Omega u_-(t)\vap \,dx$.}
We claim that
\be{lowint}
\int_\Omega u_-(t)\vap \,dx\le C,\quad t<0.
\ee

 The function  $v=-u$ is a classical solution of the corresponding problem with absorbing gradient term:
 \be{eqHJabs}
 \begin{cases}
	v_t-\Delta v+|\nabla v|^p=0,&(x,t)\in \R^n_+\times (-\infty,0),\\
	\noalign{\vskip 1mm}
	v=0,&(x,t)\in \partial\R^n_+\times (-\infty,0).\\
	\end{cases}
\ee
Let $\eps\in(0,1)$ and set $z_\eps(t):=\int_{\Omega_\eps} v_+(x,t)\vap(x) \,dx$.
By the parabolic Kato inequality (cf.~Lemma~\ref{pKi}), applied over $\Omega_\eps$
with test-function $\vap$, and using $\varphi=0$ on $\{x_n>\eps\}\cap\partial\Omega_\eps$, we have 
\begin{align*}
z_\eps'(t)\le - \int_{\Omega_\eps}\vap|\nabla v_+|^p\,dx-\int_{\Omega_\eps} \nabla v_+ \nabla\vap \,dx+\int_{\{x_n=\eps\}}  \chi_{\{v>0\}}v_\nu\vap \,d\sigma.
 \end{align*}
Recalling the notation \eqref{defIJeps} and using \eqref{euphi2a},  \eqref{euphi2} and $\varphi\le C\psi$, we obtain 
$$\begin{aligned}
 z_\eps'(t)
 &\le -\frac12 \int_{\Omega_\eps} |\nabla v_+|^p\vap \,dx+C+I_\eps \\
& \le -C_1\int_{\Omega_\eps} v_+^p \vap \,dx+C+I_\eps+J_\eps
  \le -C_1 z_\eps^p+C+I_\eps+J_\eps.
   \end{aligned}$$
 Fix $t_0<t_1<0$. By \eqref{boundIeps}, it follows that
 $$z'_\eps(t)\le  -C_1 z_\eps^p + C,\quad t_0<t<t_1,\ \eps\in(0,\eps_0),$$
with $\eps_0\in(0,1)$ depending on $u,t_0,t_1$.
It follows (cf.~Lemma~\ref{blowupineq}(ii)) that
 $$z_\eps(t)\le C_2\big[1+(t-t_0)^{-\beta}\big],\quad t_0<t<t_1,\ \eps\in(0,\eps_0).$$
 Letting $\eps\to 0$ and then $t_0\to-\infty$, $t_1\to 0$, this implies claim \eqref{lowint}.
 
 \smallskip

 {\bf Step 3.} {\it Conclusion.}
Writing $u_+=u+u_-$, inequalities \eqref{upint}  and \eqref{lowint}  
 imply $\int_\Omega u_+(t)\vap \,dx\le C\big(1+|t|^{-\beta}\big)$, hence
 \eqref{1integralcontrol} and \eqref{vap3}. 
 \end{proof}

\subsection{Pointwise estimates and proof of Theorem~\ref{thmancient3}(ii)} 
 
 We shall here prove Theorem~\ref{thmancient3}(ii) along with the following lower pointwise estimate of $u$,
 which is a preliminary step to the proof of positivity  (Theorem~\ref{thmancient3}(i)).
 In this subsection $C, C_i$ denote generic positive constants depending only on $n,p$.

  \begin{lem}\label{1L3b}
  Let $p>2$ and let $u$ be an ancient solution of \eqref{eqE1}. 
Then
 \be{vap3b}
u(x,t) \ge -Cx_n^{1-\beta}\quad\hbox{in $Q$}.
\ee
 \end{lem}
 
 \begin{proof}[Proof of Theorem~\ref{thmancient3}(ii) and Lemma~\ref{1L3b}]

 {\bf Step 1.}{\hskip 1mm}{\it Upper estimate for $x_n=1$.} 
 Fix $t<0$, $x'\in \R^{n-1}$, $Z=(x',1)$ and set $\hat{B}=B_{1/2}(Z)$.
 Pick any $s\in (t,0)$ and set $\tau=s-t$.
By Theorem~\ref{thmLY}(ii) and a shift in time and space, we have 
 $$u(Z,t)\le u(y,s)+C\big(\tau^{-\beta}+\tau\big),\quad y\in \hat{B}.$$
Using \eqref{1integralcontrol} in Proposition \ref{1L3}, we then obtain
$$\begin{aligned}
  u(Z,t)
  &\le \inf_{y\in \hat{B}} u(y,s)+C\big(\tau^{-\beta}+\tau\big)\\
& \le C\Big(\|u_+(s)\|_{L^1(\hat{B})}+\tau^{-\beta}+\tau\Big)
  \le C\big(|s|^{-\beta}+\tau^{-\beta}+\tau\big)
   \end{aligned}$$
  Choose $s=t/2$, hence $\tau=|t|/2$, if $t\in[-1,0)$ and $s=t+1$, hence $\tau=1$, if $t<-1$.
Taking supremum over $x'\in\R^{n-1}$, we then obtain
 \be{1l11}
u(x',1,t)\le C\big(1+|t|^{-\beta}\big),\quad x'\in\R^{n-1},\ t<0.
 \ee
 
 \smallskip

 {\bf Step 2.} {\it Lower estimate for $x_n=1$.}
With the notation of Step 1, by Theorem~\ref{thmLY}(ii) and a shift in time and space, we have 
 $$u(Z,t)\ge u(y,t-1)-C,\quad y\in \hat{B}.$$
Using \eqref{vap3} in Proposition \ref{1L3}, we then obtain
 $$u(Z,t)\ge \sup_{y\in \hat{B}} u(y,t-1)-C \ge -C\Big(\|u_-(t-1)\|_{L^1(\hat{B})}+C\Big)\ge -C.$$
hence
 \be{1l11b}
u(x',1,t)\ge -C,\quad x'\in\R^{n-1},\ t<0.
 \ee
 
 \smallskip

 {\bf Step 3.} {\it Scaling argument.}
For given $\lambda>0$, set $u_\lambda(x,t)=\lambda^{\beta-1}u(\lambda x , \lambda^2 t)$.
Since $u_\lambda$ is also an ancient solution of \eqref{eqE1} and $C$ in \eqref{1l11}, \eqref{1l11b} depends only $n,p$,
we may apply these estimates to $u_\lambda$, which yields
$$ \lambda^{\beta-1}u(x',\lambda, t)=u_\lambda(\lambda^{-1}x',1,\lambda^{-2}t)\in \big[-C, \,C(1+\lambda^{2\beta} |t|^{-\beta})\big],$$
for all $x'\in\R^{n-1}$, $ t<0$, $\lambda>0$.
Taking $\lambda=x_n$, we obtain \eqref{controludsqa} and \eqref{vap3b}.

\smallskip
 {\bf Step 4.} {\it Proof of the gradient estimate \eqref{controludsq1a}.}
Fix $(x,t)\in Q$ and let $R=x_n/2$.
Pick any $t_0\in(-\infty,t)$ and set $M(t_0)=\underset{B_R(x)\times[t_0,t]} \sup |u|$. 
 Applying \eqref{grabound} to the solution $v(x,s)=u(x,s+t_0)$, defined in $B_R(x)\times(0,t-t_0]$, at $s=t-t_0$,
we have
$$\begin{aligned}
|\nabla u(x,t)|
&\le C\Big(\frac{M(t_0)}{R}+\Big(\frac{M(t_0)}{R^2}\Big)^{1/p}+\Big(\frac{M(t_0)}{t-t_0}\Big)^{1/p}\Big)\\
&\le C\Big(\frac{M(t_0)}{x_n}+\Big(\frac{M(t_0)}{x_n^2}\Big)^{1/p}+\Big(\frac{M(t_0)}{t-t_0}\Big)^{1/p}\Big)\\
&\le C\Big(\frac{M(t_0)}{x_n}+x_n^{-\beta}+\Big(\frac{M(t_0)}{t-t_0}\Big)^{1/p}\Big),
\end{aligned}$$
where we used Young's inequality in the last estimate.
Since, by \eqref{controludsqa} and \eqref{vap3b}, $M(t_0)\le C\big(x_n^{1-\beta}+x_n^{1+\beta}|t|^{-\beta}\big)$,
 we deduce that
 $$
|\nabla u(x,t)|\le C\left(x_n^{-\beta}+x_n^{\beta}|t|^{-\beta}+ \big(x_n^{1-\beta}+x_n^{1+\beta}|t|^{-\beta}\big)^{1/p} (t-t_0)^{-1/p}\right)
$$
and \eqref{controludsq1a} follows by letting $t_0\to-\infty$.
 \end{proof}
 
 \vspace{2pt}
 
\subsection{Proof of   Theorem~\ref{thmancient3}(i) for $p>2$}
\label{foward0}

{\bf Step 1.} {\it Preliminaries.} 
Let $u$ be any ancient solution of \eqref{eqE1}. 
The function $v:=-u$ is a solution of \eqref{eqHJabs}
and our goal is to show that~$v\le 0$.
From Lemma~\ref{1L3b}, there exists a (universal) $C_0=C_0(n,p)>0$ such that
\be{defC0}
v\le C_0x_n^{1-\beta}\quad\hbox{in } Q.
\ee
Starting from this a priori estimate, the proof will be based on comparison with a one-dimensional, forward self-similar solution $W$,
with suitable behavior at space infinity.
Specifically, we need $W$ to grow faster than \eqref{defC0} as $x_n\to\infty$ for fixed $t$,
whereas $W$ will decay as $t\to\infty$ for fixed $x$.
We are thus are looking for a special solution of the form
$$W(s,t):=(t+1)^\gamma\phi(y),\quad y=\frac{s}{\sqrt{t+1}},$$
where $\gamma=\frac{p-2}{2(p-1)}\in (0,1/2)$.
By direct computation, such $W$ is a classical solution of
 \be{eqHJabs2}
 \begin{cases}
W_t-W_{ss}+|W_s|^p=0,&s\in (0,\infty),\ t\in (0,\infty),\\
	\noalign{\vskip 1mm}
	W(0,t)=0,&t\in (0,\infty)
	\end{cases}
\ee
provided the profile $\phi\in C^2([0,\infty))$ solves the ODE:
\begin{eqnarray}
\phi''&=&\gamma \phi-\ts\frac12 y\phi'+|\phi'|^p,\qquad y>0, \label{equphi1}\\
 \phi(0)&=&0. \label{equphi1b}
\end{eqnarray}
For problem \eqref{equphi1}-\eqref{equphi1b}, we have the following result, whose proof is postponed to section~\ref{SecFwd}.

\begin{prop}\label{positivity1}
 Let $p>2$.
There exists $\alpha>0$ such that the solution $\phi\in C^2([0,\infty))$ of \eqref{equphi1}-\eqref{equphi1b}
is global positive increasing and satisfies
 \be{asymptphi}
 \phi(y)\ge \alpha y,\quad \hbox{for all } y>0.
 \ee
\end{prop}

\begin{remark} \rm
We shall show  in Section~\ref{SecFwd} that $\alpha$ is actually unique, and
additional information on problem \eqref{equphi1}-\eqref{equphi1b} will be given,
including precise asymptotic behavior of the solution in Proposition~\ref{positivity1}.
We refer to \cite{GK,FK} for related results on the ODE \eqref{equphi1}.
In particular \cite{GK} (resp., \cite{FK}) shows the existence of negative (resp., positive) solutions to \eqref{equphi1} with the
Neumann conditions $\phi'(0)=0$ and give information on the possible asymptotic behaviors of $\phi$.
These results thus provide the existence of positive and negative, forward self-similar solutions of \eqref{eqE0} on the whole real line.
\end{remark}

\smallskip

{\bf Step 2.} {\it Universal bound.} We claim that, there exists $K=K(n,p,\alpha)>0$ such that for  
\be{boundvK}
v(x,t)\le K\quad\hbox{in } Q.
\ee
Namely, for $C_0$ given by \eqref{defC0}, we will show \eqref{boundvK} with
\be{defKC0}
K=C_0\vee \sup_{s\ge 0} \big(C_0s^{1-\beta}-\alpha s\big).
\ee

  Fix $t_0<t_1:=-(\alpha/K)^{2(p-1)}$, 
   hence 
$$  B:=K^{p-1}\alpha^{1-p} (1-t_0)^{1/2}>1,$$
 and let $w(x,t):=W(x_n,t-t_0)+K$.
By \eqref{defC0}, \eqref{defKC0}, we have
\be{compwv1}
w(x,t)\ge C_0 \ge v(x,t) \quad\hbox{in } (\R^{n-1}\times[0,1])\times(t_0,0).
\ee
Next, we deduce from \eqref{asymptphi} and $\gamma=(1-\beta)/2$ that
\be{lowerw}
w(x,t)\ge \alpha (t-t_0+1)^{-\beta/2}x_n+K\quad\hbox{in }\R^n_+\times[t_0,\infty).
\ee
It follows from \eqref{defC0}  and \eqref{lowerw} with $t=t_0$ that
\be{compwv2}
w(x,t_0)\ge v(x,t_0) \quad\hbox{in } \R^n_+.
\ee
 Also, 
we have $\alpha (t-t_0+1)^{-\beta/2}s+K\ge  C_0s^{1-\beta}$ for all $s\ge B$
and $t\in(t_0,0)$.
This along with \eqref{defC0} and \eqref{lowerw} implies
\be{compwv3}
w(x,t)\ge v(x,t) \quad\hbox{in } (\R^{n-1}\times[B,\infty))\times(t_0,0).
\ee
Now set $Pw:=w_t-\Delta w+|\nabla w|^p$ and $D:=\R^{n-1}\times(1,B)\subset \R^n$. 
By \eqref{compwv1},  \eqref{compwv2} and  \eqref{compwv3} we have
$$
\begin{cases}
Pv=Pw=0,&x\in D,\ t\in(t_0,0),\\
v(x,t)\le w(x,t),&x\in \partial D,\ t\in(t_0,0),\\
v(x,t_0)\le w(x,t_0),&x\in D.\\
\end{cases}
$$
Moreover, $v,w,|\nabla w|$ are bounded in $D\times (t_0,0)$ owing to \eqref{defC0} and the form of $W$.
We may thus apply the comparison 
 principle for unbounded domains (cf.~\cite[Proposition~52.6]{QSb19}),  
which gives $v\le w$ in $D\times(t_0,0)$, hence in $\R^n_+\times(t_0,0)$ in view of \eqref{compwv1} and \eqref{compwv3}.
Let $M=\sup_{s\in[0,1]}\phi'(s)$ and recall $\gamma=(1-\beta)/2$. 
For fixed $(x,t)\in \R^n_+\times(-\infty,0)$ and all $t_0<\min(t_1,t-x_n^2+1)$,  
we then have
$$v(x,t)\le w(x,t)=K+(t-t_0+1)^\gamma \phi\Big(\frac{x_n}{\sqrt{t-t_0+1}}\Big)\le
K+M(t-t_0+1)^{-\beta/2}x_n.$$
Letting $t_0\to-\infty$, we obtain \eqref{boundvK}.

\smallskip

{\bf Step 3.} {\it Scaling argument and conclusion.} 
For given $\lambda>0$, set $u_\lambda(x,t)=\lambda^{\beta-1}u(\lambda x , \lambda^2 t)$.
Since $u_\lambda$ is also an ancient solution of \eqref{eqE1} and $K$ in \eqref{boundvK} depends only $n,p,\alpha$,
we may apply this estimate to $v_\lambda=-u_\lambda$. This yields
$$v(x,t)=\lambda^{1-\beta}v_\lambda(\lambda^{-1}x,\lambda^{-2}t)\le K\lambda^{1-\beta},\quad (x,t)\in Q,$$
and the conclusion follows upon letting $\lambda\to 0$.
\qed

 \subsection{Proof of  Theorem~\ref{thmancient3}(i) for $p=2$ and partial positivity result for $p<2$} 
 
 In this subsection, we shall prove Theorem~\ref{thmancient3}(i) for $p=2$, along with the following 
 partial positivity result for $p<2$. 

  \begin{prop} \label{thmancient3subquadr}
Let $p\in(1,2)$, $u$  be an ancient solution of \eqref{eqE1} and assume that, for each $\eps>0$,
\be{hypsubquad}
\inf_{\Gamma_R\times(-\infty,-\eps]} u=o(R)\quad\hbox{as $R\to\infty$}.
\ee
Then $u\ge 0$.
 \end{prop}

Their proof is completely different from the case $p>2$, based on comparison arguments which make essential use of the at most quadratic 
growth of the nonlinearity.
By the same line of arguments, we will obtain the following estimate of $u$ near the boundary,
 that will be used in the proof of Theorem~\ref{thmancient} for $p\le 2$.
  
 \begin{lem}\label{lemsubquadr}
Let $p\in(1,2]$, $u$ be an ancient solution of \eqref{eqE1}, $\eps>0$
and assume that
\be{defM1sub}
M_1:=\sup_{\,\Gamma_1\times(-\infty,-\eps]} |u|<\infty.
\ee
Then $|u|\le Cx_n$ in $\Gamma_1\times(-\infty,-\eps]$
for some $C=C(n,p,M_1)>0$.
 \end{lem}

\begin{proof}[Proof of Theorem~\ref{thmancient3}(i) for $p=2$,
of Proposition~\ref{thmancient3subquadr} and of Lemma~\ref{lemsubquadr}.]

\phantom{a} 

{\bf Step~1.} {\it Auxiliary function.}
For $R>0$, consider the annulus $\mathcal{A}_R=\{x\in\R^n;\, \ts R<|x|<2R\}$
and let $U=U_R$ be the solution of 
\be{defUcomp}
\left.\qquad{\alignedat2
 U_t-\Delta U &= 0,    &\qquad& x\in \mathcal{A}_R,\ t>0\\
             U(x,t) &= 0,  &\qquad& |x|=R,\ t>0 \\
             U(x,t) &= 1,  &\qquad& |x|=2R,\ t>0 \\
             U(x,0) &= 1,  &\qquad& x\in \mathcal{A}_R. \\
  \endalignedat}\qquad\right\} 
\ee
By scaling we see that $U_R(x,t)=U_1(R^{-1}x,R^{-2}t)$ and, by
standard heat kernel estimates (see, e.g.,~\cite[Proposition~48.7*]{QSb19}),
we have $|\nabla U_1(y,s)|\le Cs^{-1/2}$, hence $|U_1(y,s)|\le s^{-1/2}(|y|-1)$ for $y\in \mathcal{A}_1$ and $s\in(0,1]$,
with $C=C(n)>0$.
Consequently, 
\be{boundUcomp}
0\le U_R(x,R^2)\le C(n)R^{-1}(|x|-R),\quad x\in \mathcal{A}_R.
\ee

{\bf Step 2.} {\it Proof of Proposition~\ref{thmancient3subquadr}.}
Assume $p\in(1,2)$ and \eqref{hypsubquad}. Fix $z=(z',z_n)\in\R^n_+$ and $t_0\le-\eps$.
For given $R>z_n$, set 
\be{defcomp1}
\left.\begin{aligned}
&b=(z',-R),\quad \mathcal{A}=\big\{x\in\R^n;\, \ts R<|x-b|<2R\big\},\\
\noalign{\vskip 2mm}
 &\Omega=\mathcal{A}\cap\R^n_+\subset \Gamma_R,\quad  J=(t_0-R^2,t_0],\quad D=\Omega\times J.\end{aligned}
\right\}
\ee
Note that the inner ball of $\mathcal{A}$ is tangent to $\partial\R^n_+$ from outside at the point $a=(z',0)$,
and that $\partial\Omega=S_1\cup S_2$, where
\be{defcomp2}
S_1=\big\{x\in\R^n;\, |x-b|\le 2R,\ x_n=0\big\},\ \, S_2= \big\{x\in\R^n;\, |x-b|=2R,\ x_n>0\big\}. 
\ee
We then define the lower comparison function
\be{deftildeU}
\tilde U(x,t):=K U_R(x-b,t+R^2-t_0),\quad(x,t)\in \bar D,
\ee
with $K=m_R:=\inf_{\Gamma_R\times(-\infty,-\eps]} u\in(-\infty,0]$.
We observe from \eqref{defUcomp} that $u=0\ge \tilde U$ on $S_1\times J$
and $u\ge m_R= \tilde U$ on $(S_2\times J)\cup(\Omega\times\{t_0-R^2\})$.
It follows from the maximum principle that $u\ge \tilde U$ in $D$
hence, by \eqref{boundUcomp},
\be{comptildeU}
u(z,t_0)\ge \tilde U(z,t_0)=K U_R(z-b,R^2)\ge C(n) R^{-1}m_Rz_n.
\ee
Proposition~\ref{thmancient3subquadr} then follows from assumption \eqref{hypsubquad} by letting $R\to \infty$.

\smallskip

{\bf Step 3.} {\it Proof of Theorem~\ref{thmancient3}(i) for $p=2$.}
The function $v:=e^u-1$ satisfies
$v_t-\Delta v=0$ with $v=0$ on $\partial\R^n_+$ and $v>-1$ in $Q$.
By the comparison argument in Step~2 applied with $u$ replaced by $v$
and $K=-1$ in the definition \eqref{deftildeU} of $\tilde U$, we obtain 
$v(z,t_0)\ge -C(n) R^{-1}z_n$ instead of \eqref{comptildeU}, hence
$v\ge 0$ upon letting $R\to \infty$. Therefore $u\ge 0$ and Theorem~\ref{thmancient3} for $p=2$ is proved.
\smallskip

{\bf Step 4.} {\it Proof of Lemma~\ref{lemsubquadr}.}
By the comparison argument in Step~2 applied with $R=1$ and $K=-M_1$
(where $M_1$ is defined in \eqref{defM1sub}), 
formulas \eqref{deftildeU} and \eqref{comptildeU} give
\be{lowerUsub}
u\ge -C(n)M_1 x_n, \quad x\in\Gamma_1,\ t\le -\eps. 
\ee
To show the upper estimate, fix $z=(z',z_n)\in\Gamma_1$ and $t_0\le-\eps$,
and set $v=e^{-t}(e^u-1)$. 
Since $p\le 2$, we have
 $$\begin{aligned} 
 v_t-\Delta v
&=e^{-t}e^u\big(u_t-\Delta u-|\nabla u|^2\big)-v\\
&=e^{-t}e^u\big(|\nabla u|^p-|\nabla u|^2\big)- v\le e^{-t}e^u- v=0.
 \end{aligned}$$
We take the upper comparison function
$$
\tilde U(x,t):=KU_1(x-b,t+1-t_0),\quad K=e^{1-t_0+M_1}, \ b=(z',-1).
$$
With the notation \eqref{defcomp1}-\eqref{defcomp2} with $R=1$, 
using \eqref{defUcomp}, we have $v=0\le \tilde U$ on $S_1\times J$
and $v\le K=\tilde U$ on $(S_2\times J)\cup(\Omega\times\{t_0-1\})$.
It follows from the maximum principle that $v\le \tilde U$ in $D$
hence, by \eqref{boundUcomp},
$$
u(z,t_0)\le e^{t_0}v(z,t_0)\le e^{t_0}\tilde U(z,t_0)=Ke^{t_0}U_1(z-b,1)\le C(n) e^{1+M_1}z_n.
$$
This combined with \eqref{lowerUsub} completes the proof of Lemma~\ref{lemsubquadr}.
\end{proof}
 
 \section{Proof of classification results  in half-space (Theorems~\ref{thmentire} and \ref{thmancient})} 
 \label{SecClass}

 \subsection{Proof of Theorem \ref{thmancient}}

 Let $u$ be an ancient solution of \eqref{eqE1} satisfying \eqref{hypthmancient}.
It suffices to prove that $u(x,t)=u(x_n)$ on $Q=Q_\eps:=\R^n_+\times(-\infty,-\eps]$ for each $\eps>0$.
Thus fix $\eps>0$.

\smallskip

{\bf Step 1.} {\it Uniform decay of $\nabla u$ and $u_t$ as $x_n\to\infty$.} 
We first claim that there exists a nonincreasing function $\phi: [1,\infty)\to[0,\infty)$ such that
\be{cond1}
 |\nabla u(x,t)|\le \phi(x_n),\ \ (x,t)\in \R^{n-1}\times[1,\infty)\times(-\infty,-\eps],
\ \ \hbox{with }
\lim_{s\to \infty}\phi(s)=0.
 \ee
By assumption \eqref{hypthmancient}, we have 
$$\lim_{A\to\infty}\phi_0(A)=0,
\qquad\hbox{where } \phi_0(A)=A^{-1}\sup_{\Gamma_A\times(-\infty,-\eps]} |u|.$$ 
For each $(x,t)\in \R^{n-1}\times[1,\infty)\times(-\infty,0)$, applying Theorem~\ref{propBern} to the function $v(y,s)=u(y+x,s+t-T)$
for $(y,s)\in Q_{R,T-t}$ with $R=x_n/2$ and $T>0$, we get
$$
|\nabla u(x,t)|=|\nabla v(0,T)|\le 
C(n,p) \bigg\{\phi_0({\ts\frac{3R}{2}})+\bigg(\frac{\phi_0(\frac{3R}{2})}{R}\bigg)^{\frac1p}+ \bigg(\frac{R\phi_0(\frac{3R}{2})}{T}\bigg)^{\frac1p}\bigg\}. 
$$
Letting $T\to\infty$, we deduce \eqref{cond1} 
with $\phi(s):=C(n,p)\big(s^{-\beta}+\sup_{\tau\ge s}\phi_0(3\tau)\big)$.

We next claim that there exists a  nonincreasing function $\tilde\phi: [1,\infty)\to[0,\infty)$ such that
 \be{cutancienb}
|u_t(x,t)| \le \tilde\phi(x_n),\ \ (x,t)\in \R^{n-1}\times[1,\infty)\times(-\infty,-\eps],
\ \ \hbox{with }
\lim_{s\to \infty}\tilde\phi(s)=0.
 \ee
Let $(\xi,t_0)\in\R^{n-1}\times(-\infty,-\eps]$, $A\ge 1$ and $D:=B'_1(\xi)\times(A-\ts\frac12,A+\ts\frac12) \times[t_0-2,t_0]$.
Set $$m=\inf_{D} u,\quad M=\sup_{D} u,\quad w=u-m.$$ 
By \eqref{cond1}, we have  $M-m\le 2\|\nabla u\|_{L^\infty(D)}\le f(A):=\underset{s\ge A-1/2}{\sup}\phi(s)\to 0$ as $s\to \infty$.
 Therefore,  
$$0\le w\le M-m\le f(A)\quad \hbox{and } \quad |\nabla w|=|\nabla u|\le f(A)\quad \hbox{in } D.$$
Setting $D':=B'_{1/2}(\xi)\times(A-\ts\frac14,A+\ts\frac14)\times [t_0-1,t_0]$, for each $q\in(1,\infty)$, by  
 interior $L^q$ parabolic regularity applied to the equation $w_t-\Delta w=|\nabla w|^p$, we have
 $$\begin{aligned}
 \|w\|_{W^{2,1;q}(D')}
 &\le C(n,q)(\|w\|_{L^\infty(D')}+\||\nabla w|^p\|_{L^\infty(D')})\\
 &\le \hat f(A):=C(n,q)(f(A)+f^p(A)).
  \end{aligned}$$
Then taking $q=q(n)$ large and using standard imbeddings, we get
$$\||\nabla w|^p\|_{C^{1/2,1/4}(D')}\le C\|w\|^p_{C^{3/2,1/4}(D')}\le C\|w\|_{W^{2,1;q}(D')}\le C\hat f(A)$$
wi!th $C=C(n,p)>0$
hence, by interior Schauder parabolic regularity,
$$|w_t(\xi,t_0)|\le C(n)(\|w\|_{L^\infty(D')}+\||\nabla w|^p\|_{C^{1/2,1/4}(D')})\le \tilde \phi(A):= C(n,p)\hat f(A).$$
Since $w_t=u_t$, this yields \eqref{cutancienb}.

 On the other hand, we shall use the following estimates:
\be{cond2}
|u(x,t)|\le Cy^\kappa,\quad (x,t)\in \Gamma_1\times(-\infty,-\eps],
 \ee
\be{cond2grad}
 |\nabla u(x,t)|\le C(y^{-\beta}+y^{1-\kappa}),\quad (x,t)\in \R^n_+\times(-\infty,-\eps],
 \ee
  where $\kappa=1-\beta>0$ if $p>2$, $\kappa=1$ if $p\in(1,2]$,
 and the constant $C>0$ may depend on $u$ and $\eps$.
For $p>2$, these estimates follow from Theorem~\ref{thmancient3}(ii).
  For $p\in(1,2]$, they follow from
   Lemma~\ref{lemsubquadr}, Theorem~\ref{propBern}
and assumption \eqref{hypthmancient} (which in particular guarantees that $\sup_{\,\Gamma_1\times(-\infty,-\eps]} |u|<\infty$).

\smallskip 

{\bf Step 2.} {\it Translation-compactness argument.}
  In this step, we write $(x,t):=(x',y,t)\in \R^{n-1}\times[0,\infty)\times(-\infty,-\eps]$.
  We shall use translations parallel to the boundary and translations in time.
Namely, fixing $h\in \R^{n-1}\setminus\{0\}$ (resp.,~$h<0$),
we set
\be{defv}
v(x,t)=u(X_h)-u(x,t), \ \hbox{ where }
 X_h=(x'+h,y,t) \ \ \big(\hbox{resp., }  X_h=(x,t+h)\big).
\ee
 It suffices to show that $v\equiv 0$. We note that $|v|\le C(n,p,\eps)$ for $y\le 1$, owing to \eqref{cond2}.
On the other hand, for $y\ge1$, by \eqref{cond1} (resp., \eqref{cutancienb}) we have $|v|\le C|h|$ 
with $C=\phi(1)$ or $\tilde\phi(1)$. Consequently, 
$$\sup_Q |v|<\infty.$$

Assume for contradiction that $v\not\equiv 0$. We may assume
 \be{supanc}
 \sigma:=\sup_Qv>0
 \ee
 (the case $ \inf_Q v<0$ being similar). 
 It follows from  \eqref{cond1}, \eqref{cutancienb}, \eqref{cond2} and \eqref{defv} that there exist 
 $0<\delta<1<A$ such that 
 \be{supanc2}
 |v|\le \sigma/2 \quad\hbox{ in $\R^{n-1}\times([0,\delta]\cup[A,\infty))\times(-\infty,-\eps]$.}
 \ee
 Therefore, 
 there exists a sequence $(x'_j,y_j,t_j)_j \in \R^{n-1}\times[\delta,A]\times(-\infty,-\eps]$ 
 such that 
\be{cvg1anc}
  \lim_{j\to \infty} v(x'_j,y_j,t_j)=\sigma
  \ee
 (covering either possibility whether or not the supremum is attained),
and we may assume that $y_j\to y_\infty\in [\delta,A]$. 

Next define 
$$u_j(x',y,t)= u(x'+x'_j,y,t+t_j),\quad (x',y,t)\in Q':=\R^{n-1}\times[\delta,A]\times(-\infty,0],$$
 and note that 
 \be{4.8}
 \sup_{(x,t)\in Q'}  \big(u_j(X_h)-u_j(x,t)\big)
 \le\sup_{(x,t)\in Q'} v=\sigma.
 \ee
By \eqref{cond2}, \eqref{cond2grad} and interior parabolic estimates, it follows that $(u_j)_j$  is relatively compact in $C^{2,1}_{loc}(Q')$.
Therefore, (some subsequence of) $(u_j)_j$ converges in that topology to a solution $U\in C^{2,1}(Q')$ of $U_t-\Delta U=|\nabla U|^p$.
From \eqref{cvg1anc}, we have 
\be{lim1anc}
U(h,y_\infty,0)-U(0,y_\infty,0)=\sigma
 \ \ \big(\hbox{resp., } U(0,y_\infty,h)-U(0,y_\infty,0)=\sigma\big).
\ee
Since, by  \eqref{supanc2}, 
$|v|\le \sigma/2$ in $\R^{n-1}\times\{\delta,A\}\times(-\infty,-\eps]$,
this implies that 
\be{yinftyint}
y_\infty\in (\delta,A).
\ee
Moreover, by \eqref{cond2grad}, we have 
\be{boundDU}
\sup_{Q'}|\nabla U|<\infty.
\ee

Put now 
$$V(x,t)=U(X_h)-U(x,t).$$
By \eqref{4.8}, for each $(x,t)\in Q'$, we have $V(x,t)=\lim_{j\to\infty} \big(u_j(X_h)-u_j(x,t)\big)
\le \sigma$, so that \eqref{lim1anc} implies $V(0,y_\infty,0)=\sigma=\sup_{Q'} V$. But $V$ satisfies 
$$V_t-\Delta V=B\cdot \nabla V\quad\hbox{ in $Q'$},$$
where 
$$B(x,t):=\int_0^1 G\Big(s\nabla U(X_h)+(1-s)\nabla U(x,t)\Big)ds,\quad G(\xi)=p|\xi|^{p-2}\xi.$$
In view of \eqref{yinftyint}, and since also $B$ is bounded in $Q'$ owing to \eqref{boundDU},
it follows from the strong maximum principle that $v\equiv\sigma$ in $\overline{Q'}$.
But this contradicts  \eqref{supanc2}.
Consequently, $v\equiv 0$ and the proof is complete.
\qed

 \subsection{Proof of classification of entire solutions (Theorem~\ref{thmentire})} 
For each $t\in(-\infty,0)$ and $T>0$, since the function $v(x,s)=u(x,s+t+T)$ for $(x,s)\in \R^n_+\times(-\infty,0)$ is an ancient solution 
 of \eqref{eqE1},
it follows from Theorem~\ref{thmancient3}(ii) that
$u(x,t)=v(x,-T)\le C(n,p)\big(x_n^{1-\beta}+x_n^{1+\beta}T^{-\beta}\big)$.
Upon letting $T\to\infty$, we deduce that
$$u(x,t)\le C(n,p)x_n^{1-\beta},\quad (x,t)\in \R^n_+\times\R.$$
For each $\tau>0$, the ancient solution $u_\tau(x,t):=u(x,t+\tau)$
thus satisfies the assumption of Theorem~\ref{thmancient}, and the latter guarantees
 that $u(x,t+\tau)$ depends only on $x_n$ for $t<0$.
Since this holds for any $\tau>0$, we conclude that $u=u(x_n)$ in $\R^n_+\times\R$.
\qed
    
 \section{Backward self-similar solutions: proof of Theorem~\ref{thmancient2}}
  \label{SecBwd}

We may assume $n=1$. 
In this section we assume $p>1$, we 
set $\gamma=\frac{p-2}{2(p-1)}=\frac{1-\beta}{2}\in (-\infty,\ts\frac12)$ 
and let
\be{defuphi}
u(x,t)=|t|^\gamma \phi(y),\quad y=x/\sqrt{|t|},\quad\hbox{ for  $x>0,\ t<0$}.
\ee
The existence of an ancient solution of \eqref{eqE1} of the form \eqref{defuphi} is reduced to the initial value problem
\be{ext}
\begin{cases}
\phi''=\frac12 y\phi'-\gamma \phi-|\phi'|^p,\quad y>0\\
 \phi(0)=0\\
\phi'(0)=\alpha>0.
\end{cases}
\ee
We split the proof of Theorem \ref{thmancient2} into the following two propositions,
which are respectively concerned with the existence and with the asymptotic behavior.
We denote by $[0,R_{max})$, with $R_{max}=R_{max}(\alpha)$, the maximal existence interval of $\phi$.

  \begin{prop}\label{existence}
 There exists $\alpha_0>0$ such that, for any $\alpha\in(0,\alpha_0)$,
there holds $R_{max}=\infty$,
$\phi, \phi'>0$ on $(0,\infty)$ and there exists $\bar R>0$ such that
$$\hbox{$\phi''(y)<0$ on $[0,\bar R)$\quad and\quad $\phi''(y)>0$ on $(\bar R,\infty)$}.$$
 \end{prop}
  
  \begin{proof}[Proof of Proposition \ref{existence}]
{\bf Step 1.} 
Set
\be{defeps}
\bar\gamma=\max(\gamma,0)\in[0,\ts\frac12),\quad
\eps=\ts\frac12-\bar\gamma\in (0,\ts\frac12).
\ee
We claim that, for any $\alpha>0$,
  \be{condR}
  \bar R:=\sup\Big\{r\in(0,R_{max}),\ \phi''<0 \ \hbox{and} \ \phi'>(1-\eps)\alpha\  \hbox {on  $(0,r)$}\Big\}
\in(0,\infty).
\ee  
Since $\phi''(0)=-\alpha^p<0$ owing to \eqref{ext}, we have $\bar R>0$.
We cannot have
 $\bar R=\infty$, since then
  $\ell:=\lim_{y\to \infty}\phi'(y)\in(0,\infty)$, hence
 $$|\phi'(y)|^p=\ts\frac12 y\phi'-\gamma \phi-\phi''(y)\ge\ts\frac12 y\phi'-\gamma \phi 
 \sim (\ts\frac12-\gamma)\ell y= \ts\frac{\beta}{2}\ell y\to\infty,\ \hbox{as } y\to\infty{\hskip -2pt}:$$
 a contradiction.
 
\smallskip

{\bf Step 2.}  Set
\be{defR1}
 R_1=\min\big(1,\sqrt{\eps/\bar\gamma}\big),\quad \alpha_0:=(\eps/2)^\beta.
 \ee
We claim that, for $\alpha\in(0,\alpha_0)$,
    \be{Rfini} 
 \bar R\le   R_1
 \Longrightarrow  \phi'(\bar R)>(1-\eps) \alpha.
  \ee
Indeed, if $\alpha<\alpha_0$
and $\bar R\le  R_1$,
then
    \be{Rfini2} 
    \ts\frac12\bar\gamma y^2+\alpha^{p-1} y<\eps\quad\hbox{on $(0,\bar R]$.}
    \ee
Since
     \be{Rfini3b} 
      0<\phi'\le \alpha
      \quad\hbox{and}\quad
     0<\phi\le \alpha y\quad \hbox{on $(0,\bar R]$,}
     \ee
     we deduce from \eqref{ext} that $\phi''\ge -\gamma \phi-{\phi'}^p\ge -\bar\gamma\alpha y-\alpha^p$ on $(0,\bar R]$.
Integrating and using \eqref{Rfini2}  we get
$$\phi'\ge \alpha\big(1-\ts\frac12\bar\gamma y^2-\alpha^{p-1} y\big)>(1-\eps) \alpha
\quad\hbox{on $(0,\bar R]$.}$$
This proves \eqref{Rfini}.

\smallskip

{\bf Step 3.} 
We claim that,
for $\alpha\in(0,\alpha_0)$,
     \be{Rfini3} 
     \phi'(\bar R)>(1-\eps)\alpha \quad\hbox{and}\quad \phi''(\bar R)= 0.
     \ee
 Indeed, if not then, by the definition of $\bar R$,
 we have $\phi'(\bar R)=(1-\eps)\alpha$ and $\phi''(\bar R)\le 0$,
 hence $\bar R> R_1$ by~\eqref{Rfini}.
Using \eqref{ext}, \eqref{defeps}, \eqref{defR1} and \eqref{Rfini3b}, 
 it follows that
 $$\begin{aligned}
 \phi''(\bar R)
 &\ge \ts\frac12(1-\eps)\alpha \bar R-\bar\gamma \alpha \bar R-\alpha^p \\
 &=\ts\frac12\big(1-2\bar\gamma-\eps- 2\alpha^{p-1}\big) \alpha \bar R
 =\ts\frac12(\eps- 2\alpha^{p-1}) \alpha \bar R> 0:
\end{aligned}$$
a contradiction.
 \smallskip
 
 {\bf Step 4.} 
We claim that 
     \be{Rfini4} 
     \phi,\phi' >0 \quad\hbox{on $(0,R_{max})$,}
     \qquad
     \phi''>0
 \quad\hbox{on $(\bar R,R_{max})$}
 \ee
 and
      \be{Rmaxinfty} 
 R_{max}=\infty.
 \ee
 
Differentiating the equation, we obtain 
     \be{eqphi3} 
     \phi'''=(\ts\frac12-\gamma)\phi'+\frac12 y\phi''-p|\phi'|^{p-2}\phi'\phi''.
     \ee
In particular, by \eqref{Rfini3}, we have $\phi'''(\bar R)>0$, hence 
$\phi''>0$ for $y>\bar R$ close to $\bar R$.
 To verify \eqref{Rfini4}, since $\phi'>0$ on $(0,\bar R]$, it suffices to show that $\phi''>0$ on $(\bar R,R_{max})$.
If not then there exists a minimal $R_1\in(\bar R,R_{max})$ such that $\phi''(R_1)=0$. Then $\phi'(R_1)>0$ and $\phi'''(R_1)\le0$.
But \eqref{eqphi3} implies $ \phi'''(R_1)=(\ts\frac12-\gamma)\phi'(R_1)>0$:
a contradiction.
 
Finally, from \eqref{ext},
setting $\tilde\gamma=\max(0,1-\gamma)$, we have 
$$\big({\phi'}^2+\phi^2\big)'=2\big(\phi''+\phi\big)\phi'\le y{\phi'}^2+2\tilde\gamma\phi\phi'\le (y+\tilde\gamma)\big({\phi'}^2+\phi^2\big)$$
for all $y\in (0,R_{max})$, hence
 ${\phi'}^2+\phi^2\le C\exp[\frac12 y^2+\tilde\gamma y]$,
 which guarantees \eqref{Rmaxinfty}.
 The proof is complete.  \end{proof}

 We state the second proposition as follows.
 \begin{prop}\label{equivalence}
 Let $\phi$ be a solution of \eqref{ext} as in Proposition \ref{existence}. Then 
 \be{limitinfty}
 \lim_{y\to \infty} \frac{\phi(y)}{y^{\beta+1}}=L:=\frac{p^{-\beta}}{\beta+1}.
 \ee
 \end{prop}
 
For the proof of Proposition \ref{equivalence}, in view of Proposition \ref{existence}, we may fix $y_0>0$ such that 
 \be{posphiphiprime}
 \phi, \ \phi',\ \phi''>0\quad \hbox{on } [y_0,\infty).
 \ee
 We rely on the following lemmas.
 
 \begin{lem}\label{lem2}
  Let $\phi$ be a solution of \eqref{ext} as in Proposition \ref{existence}.
  \smallskip
  
 (i) There exists $C>0$ such that
\be{4a}
 \phi'(y)\le Cy^\beta,\quad y\ge y_0.
\ee
  
  (ii) The function $Z(y)=y^{-1}\phi'^{p-1}$ satisfies
 \be{4}
\liminf_{y\to \infty} \ Z(y)\ge \lambda:=\strut \ts\frac12 \min(1,\beta).
 \ee
In particular there exists $C_1>0$ such that 
 \be{4'}
\phi'(y)\ge C_1 y^\beta,\quad y\ge 1.
 \ee
 \end{lem}

   \begin{proof}[Proof of Lemma \ref{lem2}] 
 (i) 
By \eqref{ext} and \eqref{posphiphiprime},   for $y\ge y_0$, we have
 $${\phi'}^p\le \ts\frac12 y\phi'+|\gamma|\phi\le \max(y\phi',2|\gamma|\phi),$$
hence
  $$\phi'\le y^\beta+(2|\gamma|)^{1/p}\phi^{1/p}.$$
  By ODE comparison, we deduce that
     $\phi\le \bar\phi :=Cy^{\beta+1}$ for some large $C>0$, and then \eqref{4a}.
     
     \smallskip
     (ii) We claim that, for some $C>0$,
          \be{claimbetaZ}
          \frac{\beta Z'}{yZ}  \ge \lambda-Z-\frac{C}{y},\quad y>y_0.
          \ee
By \eqref{ext}, we have
          \be{eqhZ}
          h:=\frac{\phi''}{y\phi'}=\frac12-\frac{\gamma \phi}{y\phi'}-Z.
          \ee
Moreover,
 $yZ'+Z=(yZ)'=(\phi'^{p-1})'=(p-1)y^2Zh$
 and, dividing by $(p-1)y^2Z$, we obtain
          \be{eqhZ2}
          \frac{\beta Z'}{yZ}=h-\frac{\beta}{y^2}
   =\frac12-Z-\frac{\gamma \phi}{y\phi'}-\frac{\beta}{y^2}.
   \ee
  If $1<p\le 2$, hence $\gamma\le 0$ and $\lambda=1/2$, this in particular yields \eqref{claimbetaZ}.
If $p>2$, hence $\gamma>0$, we deduce from \eqref{posphiphiprime} and \eqref{eqhZ} that
 $$(y(Z+h))'=\frac12 -\gamma\Big( \frac{\phi}{\phi'}\Big)'=\frac{1}{2}-\gamma+\gamma\frac{\phi \phi''}{\phi'^2}
 \ge \frac12 -\gamma=\frac{\beta}{2}=\lambda.$$
By integration, we get
 \be{3}
 Z+h\ge \lambda-\frac{C}{y}
 \ee
 and \eqref{claimbetaZ} follows from the first equality in \eqref{eqhZ2}.

Now, it follows from  \eqref{claimbetaZ}  that
$\limsup_{y\to \infty}\,Z \ge \lambda$ 
(since otherwise, $Z'/(yZ)\ge\eta>0$ as $y\to \infty$, hence $\lim_{y\to \infty}\,Z=\infty$, a~contradiction).
Therefore, if \eqref{4} fails, then there exist  $\eps>0$ and $y_i\to \infty$ such that 
$$Z(y_i)\le \lambda-\eps\quad \hbox{and }\quad Z'(y_i)=0.$$
But this contradicts \eqref{claimbetaZ}, hence \eqref{4} is proved.
 \end{proof}
 
 \begin{lem}\label{lem4}
  Let $\phi$ be a solution of \eqref{ext} as in Proposition \ref{existence}.
There exist $C, A>0$ such that
 \be{boundphi2sec5}
 0\le \phi''(y)\le C y^{\beta-1},\quad y\ge A.
 \ee
 \end{lem}
 \begin{proof}[Proof of Lemma \ref{lem4}]
 By \eqref{eqphi3} and \eqref{4a} we have
\be{phi2sec5a}
 \phi'''+(p\phi'^{p-1}-\ts\frac{y}{2})\phi''=(\ts\frac12-\gamma)\phi'\le C y^\beta.
 \ee
 Let $\eps>0$. By \eqref{4}, there exists $y_1>y_0$ such that 
 $\phi'^{p-1}\ge (\ts\frac12 \min(1,\beta)-\eps)y$ for $y\ge y_1$ hence,
 taking $\eps=\eps(p)>0$ small, 
\be{phi2sec5b}
 p\phi'^{p-1}-\ts\frac{y}{2}\ge  (\ts\frac12 \min(p-1,\beta)-p\eps)y\ge 
 \ts\frac14 \min(p-1,\beta)y,\quad y\ge y_1.
 \ee
 Setting $k:= \ts\frac18 \min(p-1,\beta)$, combining \eqref{phi2sec5a}, \eqref{phi2sec5b} and using $\phi''\ge 0$, we get
 $\big(\phi''e^{ky^2}\big)'\le Cy^\beta e^{ky^2}$.
Integrating by parts, we obtain
 \begin{align*}
 \phi''e^{ky^2}&\le C\Big( 1+\int_{y_1}^y s^{\beta-1}\big(se^{ks^2}\big)ds\Big)
 \le C\Big( 1+y^{\beta-1}e^{ky^2}\Big),\quad y\ge y_1,
 \end{align*}
hence  \eqref{boundphi2sec5}.
 \end{proof}
 
Our last lemma will be also useful in Section~\ref{SecFwd}. We therefore take $\sigma\in\{-1,1\}$
 and consider more generally the equation
 \be{ODEsigma}
 \sigma\phi''=\frac12 y\phi'-\gamma \phi-|\phi'|^p.
 \ee

 \begin{lem}\label{lem3}
 Let $\sigma\in\{-1,1\}$, $A,C>0$. Let $\phi$ be a global solution of \eqref{ODEsigma} on $[A,\infty)$ such that
 \be{hypODEsigma}
 \phi'(y)\ge C y^\beta \ \hbox{ and } \ |\phi''(y)|=o\big(y^{\beta+1}\big),\hbox{ as $y\to\infty$.}
 \ee
 Then
 \be{limphi0}
\lim_{y\to \infty} \frac{\phi(y)}{y^{\beta+1}}=L:=\frac{p^{-\beta}}{\beta+1}.
\ee
 \end{lem}

 \begin{proof}[Proof of Lemma \ref{lem3}]
 Set $\psi=\frac{\phi}{y^{\beta+1}}$.
 
 \smallskip
 
 {\bf Step 1.} 
We claim that, for any sequence $y_i\to \infty$ as $i\to \infty$, we have the implication:
 \be{7}
\psi'(y_i)=o(y_i^{-1}) \Longrightarrow \lim_{i\to \infty} \psi(y_i)=L:=\frac{p^{-\beta}}{\beta+1}.
 \ee

To this end, we compute 
 $$\psi'=\frac{\phi'}{y^{\beta+1}}-(\beta+1)\frac{\phi}{y^{\beta+2}}=\frac{y\phi'-(\beta+1)\phi}{y^{\beta+2}},$$
 hence 
 \be{8f}
 \phi(y_i)=\frac{y_i}{\beta+1}\phi'(y_i)+o\big(y_i^{\beta+1}\big).
 \ee
 Multiplying by $-\gamma$ and using \eqref{ODEsigma}, we get
 $$\sigma\phi''(y_i)=\Big(\frac12-\frac{\gamma}{\beta+1}\Big)y_i\phi'(y_i)-\phi'^p+o\big(y_i^{\beta+1}\big).$$
Dividing by $\phi'(y_i)$ and using $\frac12-\frac{\gamma}{\beta+1}=\frac{1}{p}$ and 
\eqref{hypODEsigma}, we obtain
 $$\frac{y_i}{p}-\phi'^{p-1}(y_i)=\frac{\sigma\phi''(y_i)+o\big(y_i^{\beta+1}\big)}{\phi'(y_i)}=o(y_i).$$
So we have proved that 
 $\phi'(y_i)=(p^{-\beta}+o(1)) y_i^\beta$.
 Combining with \eqref{8f} we obtain the desired result.
 
  \smallskip
  
{\bf Step 2.} We next claim that $\ell:=\lim_{y\to\infty}\psi(y)$ exists in $(0,\infty)$.
\vskip 1pt
By \eqref{hypODEsigma}, for $y$ large, we have $|\phi''|\le \frac12 {\phi'}^p$ and $\phi>0$,
and \eqref{ODEsigma} then implies
${\phi'}^{p}\le y\phi'+2|\gamma|\phi$.
By the proof of Lemma \ref{lem2}(i), it follows that $\ell_2:=\limsup_{y\to\infty}\psi(y)<\infty$.
Using  \eqref{hypODEsigma} again, we see that $\ell_1:=\liminf_{y\to\infty}\psi(y)$ 
satisfies $0<\ell_1\le \ell_2<\infty$.
Assume for contradiction that $\ell_1<\ell_2$.
Then there exist $y_i,\bar{y_i}\to \infty$ such that $\psi'(y_i)=\psi'(\bar{y_i})=0$, $\psi(y_i)\to \ell_1$ and $\psi(\bar{y_i})\to \ell_2$ as $i\to \infty$.
But Step~1 then implies $\ell_1=L$ and $\ell_2=L$: a contradiction.

  \smallskip
  
{\bf Step 3.} Conclusion.
We note that there exists a sequence $y_i\to \infty$ such that $\psi'(y_i)=o\big(\frac{1}{y_i}\big)$.
Indeed, otherwise,  there holds $|\psi'(y)|\ge \eps y^{-1}$ for all $y\ge A$, with some $\eps,A>0$.
By continuity, there exists $\sigma_0\in\{-1,1\}$ such that $\sigma_0\psi'(y)\ge \eps y^{-1}$ for all $y\ge A$,
hence $\lim_{y\to \infty}\psi(y)=\pm\infty$. But this contradicts the conclusion of~Step~2.

\vskip 1pt

Now applying Step~1 to the sequence $(y_i)_i$, we deduce $\ell=L$,
which proves \eqref{limphi0}.
 \end{proof}

\begin{proof}[Proof of Proposition \ref{equivalence}]
It is a direct consequence of  \eqref{4'}, \eqref{boundphi2sec5} and  Lemma \ref{lem3}.
\end{proof}

\section{Forward self-similar solutions: proof of Proposition~\ref{positivity1}}
  \label{SecFwd}
  
 In this section we assume $p>2$.
We now investigate the initial value following problem
\be{ODEfwd}
\begin{cases}\phi''=\gamma \phi-\frac12 y\phi'+|\phi'|^p,\quad y>0\\
 \phi(0)=0\\
\phi'(0)=\alpha>0,\end{cases}
\ee
where $\gamma=\frac{p-2}{2(p-1)}=\frac{1-\beta}{2}\in (0,\frac12)$ and $\phi$ is the profile as in section~\ref{foward0}. 
We denote by $[0,R_{max})$, with $R_{max}=R_{max}(\alpha)$, the maximal existence interval of $\phi$.
Our main aim is to prove Proposition~\ref{positivity1}.
We shall use a shooting argument based on the following two sets:
$$
J_1:=\big\{\alpha>0,\ \exists\, y\in(0,R_{\max}):\ \phi''(y)<0\big\},
\qquad 
J_2:=\big\{\alpha>0,\, R_{\max}<\infty\big\}.
$$
Although this is not used in the proof of Theorem~\ref{thmancient3}(i),
we shall also give more information on the asymptotic behavior of $\phi$ at $\infty$
and on the structure of the solutions of \eqref{ODEfwd}. Namely, we shall prove:

\begin{prop}\label{equivneg1}
There is a unique $\alpha$ such that $\phi$ enjoys the properties in Proposition~\ref{positivity1}.
Moreover, $\phi$ satisfies
\be{limphi}
\lim_{y\to \infty} \frac{\phi(y)}{y^{\beta+1}}=L:=\frac{p^{-\beta}}{\beta+1}.
\ee
Furthermore, we have 
$J_1=(0,\alpha)$ and $J_2=(\alpha,\infty)$.
\end{prop}

We first note that, for all $\alpha>0$, we have
\be{phipos}
\phi,\ \phi'>0\quad \hbox{on $(0,R_{\max})$}
\ee
(indeed, otherwise, there would exist a minimal $y_1\in (0,R_{\max})$ such that $\phi'(y_1)=0$,
hence $\phi(y_1)>0$ and $\phi''(y_1)\le 0$, contradicting \eqref{ODEfwd}).
The proof of Proposition~\ref{positivity1} relies on the following three lemmas. 

\begin{lem}\label{lemJ1}
We have $J_1\subset \big\{\alpha>0;\ R_{\max}=\infty\big\}$.
Moreover, for all $\alpha\in J_1$ and $y>0$, if $\phi''(y)<0$ then
$\phi''<0$ on $[y,\infty)$.
\end{lem}
\begin{proof}
Let $\alpha\in J_1$ and $y_0\in (0, R_{\max})$ be such that $\phi''(y_0)<0$.
Assume for contradiction that there exists a minimal $y_1\in (y_0,R_{\max})$  such that $\phi''(y_1)=0$, hence $\phi'''(y_1)\ge0$. 
Differentiating \eqref{ODEfwd}, we obtain 
\be{phi3}
\phi'''=(\gamma-\ts\frac12)\phi'-\frac12 y\phi''+p\phi'^{p-1}\phi'',
\ee
hence $\phi'''(y_1)=(\gamma-\frac12)\phi'(y_1)<0$: a contradiction. Consequently, we have $\phi''<0$ in $(y_0,R_{\max})$. 
In particular, using also \eqref{phipos}, it follows that $y'$ is bounded, hence $R_{\max}=\infty$.
\end{proof}

\begin{lem}\label{lemJ2}
There exists $\alpha_1>0$ such that $(0,\alpha_1)\subset J_1$.
\end{lem}

\begin{proof}   
Let $\eps=\frac{1}{2(p-2)}$.
Set $\alpha_2:=(1+\eps)^{-p\beta}$ and let $\alpha\in(0,\alpha_2)$. We claim that 
     \be{barR1} 
R_{max}>y_1=y_1(p):=\min\big(1, [\gamma (1+\eps)+1]^{-1}\eps\big)  \ \hbox{ and }\ 
    \phi'<(1+\eps)\alpha \ \hbox {on $(0,y_1]$.}
     \ee
     Indeed, otherwise, recalling \eqref{phipos}, there is a minimal $y_0\in(0,R_{max})$ with $y_0\le y_1$,  such that $\phi'(y_0)=(1+\eps)\alpha$.
     Then $\phi'\le(1+\eps)\alpha $ and $\phi\le(1+\eps)\alpha y$ on $[0,y_0]$ hence, by \eqref{ODEfwd},
     $$\phi''\le\gamma (1+\eps)\alpha y+(1+\eps)^p\alpha^p<[\gamma (1+\eps)+1]\alpha
     \quad\hbox{on $(0,y_0)$.}$$
 By integration, using $\phi'(0)=\alpha$, we deduce $(1+\eps)\alpha=\phi'(y_0)<\alpha+[\gamma (1+\eps)+1]\alpha y_0$,
 hence $y_0> [1+\gamma (1+\eps)]^{-1}\eps$: a contradiction. The claim follows.

Next set $\alpha_1:=\min\big\{\alpha_2,(\frac{\beta}{8}y_1(1+\eps)^{-p})^\beta\big\}$
and let $\alpha\in(0,\alpha_1)$.
We claim that there exists $y_2\in(0,y_1]$ such that $\phi''(y_2)<0$, which will prove the lemma.
Assume for contradiction that $\phi''\ge 0$ on $(0,y_1]$.
Then $\alpha\le\phi'\le(1+\eps)\alpha$ on $[0,y_1]$ in view of \eqref{barR1}.
Using \eqref{ODEfwd} we obtain
 $$
 \phi''(y_1)\le \alpha \big(\gamma(1+\eps)-\ts\frac12\big)y_1+(1+\eps)^p\alpha^p
 = \alpha \big[-\ts\frac{\beta}{4}y_1+(1+\eps)^p\alpha^{p-1}\big]<0:
$$
a contradiction.
\end{proof}

\begin{lem}
We have
\be{deftildeJ}
J_2=\big\{\alpha>0,\ \exists\ y\in[0,R_{\max}):\, {\phi'}^{p-1}(y)>y+1\big\}.
\ee
In particular, $(1,\infty)\subset J_2$.
\end{lem}

\begin{proof}
Denote $\tilde J$ the set on the right hand set of \eqref{deftildeJ}.
If $\alpha\in J_2$, then $R_{\max}<\infty$ by definition, hence $\limsup_{y\to R_{\max}} \phi'(y)=\infty$ in view of \eqref{phipos}.
Therefore, $\alpha\in\tilde J$.

Conversely, let $\alpha\in\tilde J$.
Set $h={\phi'}^p-(y+1)\phi'$. For any $y\in(0,R_{\max})$ such that $h(y)>0$, we have $\phi'(y)>1$ hence, by \eqref{ODEfwd},
$$h'=[p{\phi'}^{p-1}-(y+1)]\phi''-\phi' \ge (p-1){\phi'}^{p-1}({\phi'}^p-\ts\frac12 y\phi')-\phi'
\ge {\phi'}^p-\phi'>0.$$
Consequently, $h>0$ on $[y,R_{\max})$. Using \eqref{ODEfwd} again, it follows that $\phi''>\frac12{\phi'}^p$ on $[y,R_{\max})$, 
which implies $R_{\max}<\infty$, so that $\alpha\in J_2$.
\end{proof}

Set $J_3:=(0,\infty)\setminus(J_1\cup J_2)$. We can now give the:

\begin{proof}[Proof of Proposition~\ref{positivity1}]
By the definition of $J_1, J_2$, we have
\be{propJ3}
J_3\subset \big\{\alpha>0,\ R_{\max}=\infty \ \hbox{ and }\ \forall\,y\in[0,\infty),\,\phi''(y)\ge 0\big\}.
\ee
By Lemma~\ref{lemJ1}, we have $J_1\cap J_2=\emptyset$. 
Also, by the definition of $J_1$, \eqref{deftildeJ} and continuous dependence, $J_1$ and $J_2$ are open subsets of $(0,\infty)$. Since $(0,\infty)$ is connected, it follows that $J_3\ne\emptyset$.
Picking $\alpha\in J_3$, we have $\phi'(y)\ge \phi'(0)=\alpha$ for all $y>0$ owing to \eqref{propJ3} and $\phi$
has all the desired properties.
\end{proof}

We now proceed with the proof of Proposition \ref{equivneg1}.
In view of applying Lemma \ref{lem3}, we 
derive the following asymptotic estimates on $\phi',\phi''$.

\begin{lem}\label{lemJ4}
For all $\alpha\in J_3$, we have 
\be{limsupZA1}
\ts\frac{\beta}{2} y\le \phi'^{p-1}\le y+1,\quad y>0.
\ee
\end{lem}
\begin{proof}
The upper part follows directly from \eqref{deftildeJ}.
To derive the lower part we set $\lambda:=\beta/2$, $Z(y)=y^{-1}\phi'^{p-1}$, $h=\phi''/\phi'$
and claim that
 \be{estZA1} 
\frac{Z'}{yZ}\le (p-1)(Z-\lambda).
 \ee
To this end we compute
$$
(yZ-h)'=\Big( \frac{\phi'^p-\phi''}{\phi'}\Big)'=\Big(\frac{y}{2}-\gamma \frac{\phi}{\phi'} \Big)'\ge \frac12 -\gamma=\lambda,
$$
hence $h\le y(Z-\lambda)$ by integration. Next, writing 
 $yZ'=(yZ)'-Z\le(\phi'^{p-1})'=(p-1)yZh$ and dividing by $y^2Z$, we obtain
$\frac{Z'}{yZ}\le (p-1)\frac{h}{y}$, hence \eqref{estZA1}.

 Now assume for contradiction that the lower part of \eqref{limsupZA1} fails. 
 Then there exists $y_0>0$ such that $Z(y_0)<\lambda$.
 It follows from \eqref{estZA1} that, for all $y\ge y_0$,  $Z'(y)<0$
 hence $Z(y)\le Z(y_0)$, and then that $\frac{Z'(y)}{yZ(y)}\le -\eta$ for some $\eta>0$.
 Upon integration we get $Z(y)\le Ce^{-\eta y^2/2}$.
 But since $Z=y^{-1}\phi'^{p-1}\ge \alpha^{p-1}y^{-1}$ owing to \eqref{propJ3}, this is a contradiction.
\end{proof}

\begin{lem}\label{lem4f}
For all $\alpha\in J_3$, there exists $C>0$ such that 
\be{boundphi2}
0\le\phi''\le Cy^{\beta-1},\quad y\ge 1.
\ee
\end{lem}

 \begin{proof}[Proof of Lemma \ref{lem4f}]
By \eqref{ODEfwd}, the upper part of \eqref{limsupZA1} and $\gamma-\frac12<0$, 
there exists $C>0$ such that
\be{phi3a}
 \phi'''+\big(\ts\frac12 y-p\phi'^{p-1}\big)\phi''=\big(\gamma-\frac12\big)\phi'
 \ge -C y^\beta,\quad y\ge 1.
 \ee
 Also, by the lower part of \eqref{limsupZA1}, we have 
 $p\phi'^{p-1}\ge \frac{\beta+1}{2}y$, hence 
\be{phi3b}
-\ts\frac{\beta}{2} y\ge  \frac12 y- p\phi'^{p-1}.
\ee
 Set 
\be{phi3c}
 \begin{aligned}
 &H(y):=\phi''e^{-ky^2}-C\Lambda(y),
\quad\hbox{where 
 $k=\frac{\beta}{4}$ and } \\
 &\Lambda(y)=\int_y^\infty s^\beta e^{-ks^2}ds\sim \tilde Cy^{\beta-1} e^{-ky^2},\ y\to\infty.
  \end{aligned}
  \ee
 Recalling $\phi''\ge 0$ (cf.~\eqref{propJ3}),
 it follows from \eqref{phi3a}, \eqref{phi3b} that
 $H'(y)=\big(\phi'''-\ts\frac{\beta}{2}y\phi''+Cy^\beta\big)e^{-ky^2} \ge 0$ for $y\ge 1$.
 Consequently, the limit $\ell:=\lim_{y\to\infty}  H(y)=\lim_{y\to\infty}  \phi''e^{-ky^2}$
 exists in $[0,\infty]$. But, in view of the upper part of \eqref{limsupZA1}, we necessarily have $\ell=0$.
We deduce that $H\le 0$ in $[1,\infty)$ and the conclusion follows from \eqref{phi3c}.
  \end{proof}
 
In view of the uniqueness proof, we also need the following ODE comparison principle.
Note that it in particular implies that the solutions of \eqref{ODEfwd} are ordered as a function of $\alpha>0$.

\begin{lem}\label{ODEcomp}
Let $\mathcal{P}\phi:=\phi''-\gamma \phi+\frac12 y\phi'-|\phi'|^p$, $R>0$ and $\phi_1, \phi_2\in C^2([0,R])$.
Then
$$
\left.\begin{aligned}
&\mathcal{P}\phi_1\ge \mathcal{P}\phi_2,\quad y\in(0,R],\\
&\phi_1(0)\ge \phi_2(0),\\
&\phi'_1(0)>\phi'_2(0)
\end{aligned}
\ \right\}
\Longrightarrow
 \big(\phi_1>\phi_2\ \hbox{ and } \ \phi'_1> \phi'_2\big) \ \hbox{ in $(0,R]$.}
$$
\end{lem}

\begin{proof}
Let $H:=\phi'_1-\phi_2'$. We have $H(0)>0$. If the conclusion fails, then there is a minimal $y_0\in(0,R]$ such that $H(y_0)=0$, so $H'(y_0)\le0$. Using $\mathcal{P}\phi_1 \ge \mathcal{P}\phi_2$ and $\phi'_1(y_0)=\phi'_2(y_0)$, we get
$$
H'(y_0)=\phi''_1(y_0)-\phi''_2(y_0)\ge \gamma(\phi_1(y_0)-\phi_2(y_0))=\gamma\int_0^{y_0}H(s)ds>0:
$$
a contradiction. 
\end{proof}

 \begin{proof}[Proof of Proposition \ref{equivneg1}]
It follows from  \eqref{limsupZA1}, \eqref{boundphi2} and Lemma \ref{lem3} that property \eqref{limphi} holds for any $\alpha\in J_3$.
To show the uniqueness  
of $\alpha$, assume for contradiction that $J_3$ contains at least two elements $\alpha_2>\alpha_1>0$
and denote by $\phi_1, \phi_2$ the corresponding solutions of \eqref{ODEfwd}.
Let $\delta:=\sqrt{\alpha_2/\alpha_1}>1$ and $\tilde\phi:=\delta\phi_1$.
We have 
$$\mathcal{P}\tilde\phi=\delta \big(\phi_1''-\gamma \phi_1+\ts\frac12 y\phi_1'-\delta^{p-1}{\phi'_1}^p\big)\le 0=
\mathcal{P}\phi_2\quad\hbox{in $(0,\infty)$,}$$
with $\tilde\phi(0)=\phi_2(0)=0$ and $\tilde\phi'(0)=\sqrt{\alpha_1\alpha_2}<\alpha_2=\phi'_2(0)$.
It follows from Lemma~\ref{ODEcomp} that $\phi_2>\delta\phi_1$ on $(0,\infty)$.
But since $\phi_1, \phi_2$ satisty \eqref{limphi} by Step~1, we have $\lim_{y\to \infty} [\phi_1/\phi_2](y)=1$: a contradiction.
 \end{proof}
 
 \section{Appendix}

We here collect some well-known auxiliary results used throughout the paper.
 
  \begin{lem}(Blowup and decay differential inequalities.)\label{blowupineq}
  Let $\gamma>1$, $ k>0$, $A\ge 0$, $I:=(t_0,t_1)$ and $Y\in W^{1,1}(I)$. 
  \smallskip
  
 (i) If 
  $$Y'(t)\ge k |Y(t)|^\gamma-A\quad \hbox{a.e.~in $I$}, $$
  then 
  $$Y(t)\le C(\gamma)[k(t_1-t)]^{-1/(\gamma-1)}+(2A/k)^{1/\gamma}\quad \hbox{for all } t\in I.$$
  
(ii) If 
  $$Y'(t)\le-  k|Y(t)|^\gamma+A\quad \hbox{a.e.~in $I$}, $$
  then 
  $$Y(t)\le C(\gamma [k(t-t_0)]^{-1/(\gamma-1)}+(2A/k)^{1/\gamma}\quad \hbox{for all } t\in I.$$
  \end{lem}
  
  \begin{lem}(Parabolic Kato inequality.)\label{pKi}
  Let $\Omega$ be a bounded Lipschitz domain and $v\in C^{2,1}\big(\overline{Q}_T\big)$ 
  be a solution of $v_t-\Delta v=F$ in $Q_T:=\Omega\times (0,T)$. Then $v_+$ is a weak solution of 
  $$\partial_t  v_+-\Delta v_+\le F\chi_{\{v>0\}}\quad \hbox{in } Q_T,$$
namely, 
 for every $0\le\varphi\in C^1(\overline\Omega)$ and $a.e.~t\in(0,T)$,
  $$ \frac{d}{dt}\int_\Omega v_+(t)\varphi \,dx+\int_\Omega \nabla v_+\nabla\varphi\,dx
  \le \int_\Omega F\chi_{\{v>0\}}\varphi\,dx+\int_{\partial\Omega} \chi_{\{v>0\}}v_\nu \varphi\, d\sigma.$$
    \end{lem}

\end{document}